\pdfoutput=1

\documentclass[a4paper]{article}
\usepackage{amsthm,amssymb,amsmath,enumerate,graphicx}

\newtheorem{definition}{Definition}

\newtheorem{theorem}[definition]{Theorem}
\newtheorem{corollary}[definition]{Corollary}
\newtheorem{lemma}[definition]{Lemma}

\newcommand{\FF}{{\mathbb F}}

\newcommand{\C}{{\mathcal C}}
\newcommand{\Cfin}{{\mathcal C}_{\rm fin}}
\newcommand{\CC}{{\mathbb C}}
\newcommand{\E}{{\mathcal E}}

\newcommand{\N}{{\mathbb N}}

\newcommand{\RR}{{\mathbb R}}
\newcommand{\Tbar}{{\overline T}}
\newcommand{\U}{{\mathcal U}}

\newcommand{\Z}{{\mathbb Z}}
\newcommand{\interior}{\mathaccent"7017\relax}
\newcommand{\sm}{\setminus}
\newcommand{\ke}{{\rm Ker\>}}
\newcommand{\im}{{\rm Im\>}}

\let\eps=\varepsilon
\let\sub=\subseteq
\let\subset=\subseteq

\let\phi=\varphi
\let\es=\emptyset
\newcommand\fes{(f_e)_\sharp}
\newcommand\restr{\!\restriction\!}

\newcommand{\noproof}{\unskip\nobreak\hfill\penalty50\hskip2em\hbox{}\nobreak\hfill%
       $\square$\parfillskip=0pt\finalhyphendemerits=0\par}
\newcommand{\COMMENT}[1]{}

\newcommand{\emtext}[1]{\text{\em #1}}

\newenvironment{txteq}
  {\begin{equation}\begin{minipage}[c]{0.8\textwidth}\em}
  {\end{minipage}\ignorespacesafterend\end{equation}\ignorespacesafterend}

\newcommand{\assign}{
  \mathrel{\mathop{:}}=
}

\def\lowfwd #1#2#3{{\setbox0\hbox{$#1$}\setbox1\hbox{$E'\!$}
            \mathchoice
            {{\mathop{\kern0pt #1}\limits^{\kern#2pt\raise.#3ex
     \vbox to 0pt{\hbox{$\scriptscriptstyle\rightarrow$}\vss}}}}%
            {{\mathop{\kern0pt #1}\limits^{\kern#2pt\raise.#3ex
     \vbox to 0pt{\hbox{$\scriptscriptstyle\rightarrow$}\vss}}}}%
            {\ifdim\wd0<\wd1{\,\vec{#1}\,}\else
     {\mathop{\kern0pt #1}\limits^{\kern#2pt\raise.0ex
     \vbox to 0pt{\hbox{$\scriptscriptstyle\rightarrow$}\vss}}}\fi}%
            {{\vec{#1}}}%
            }}
\def\fwd #1#2{{\lowfwd{#1}{#2}{15}}}
\def\lowbkwd #1#2#3{{\mathop{\kern0pt #1}\limits^{\kern#2pt\raise.#3ex
     \vbox to 0pt{\hbox{$\scriptscriptstyle\leftarrow$}\vss}}}}

\def\vC{\kern-1pt\fwd C3\kern-.5pt}

\def\vCC{\kern-.7pt\fwd{\C}3\kern-.7pt}
\def\vd{\kern-1pt\lowfwd d2{10}\kern-1pt}
\def\vD{\kern-.7pt\fwd D3\kern-.5pt}
\def\ve{\kern-1pt\lowfwd e{1.5}1\kern-1pt}
\def\vf{\kern-1pt\lowfwd f{1.5}1\kern-1pt}
\def\fv{\kern-1pt\lowbkwd f{1.5}1\kern-1pt}
\def\ev{\kern-1pt\lowbkwd e{1.5}1\kern-1pt}
\def\veStar{{\mathop{\kern0pt e\lower1.5pt\hbox{${}^*$}}\limits^{\kern0pt
   \raise.02ex\vbox to 0pt{\hbox{$\scriptscriptstyle\rightarrow$}\vss}}}}
\def\eStarv{{\mathop{\kern0pt e\lower1.5pt\hbox{${}^*$}}\limits^{\kern0pt
   \raise.02ex\vbox to 0pt{\hbox{$\scriptscriptstyle\leftarrow$}\vss}}}}
\def\vedash{{\mathop{\kern0pt e\lower.5pt\hbox{${}
     \scriptstyle'$}}\limits^{\kern0pt\raise.02ex
     \vbox to 0pt{\hbox{$\scriptscriptstyle\rightarrow$}\vss}}}}

\def\vE{\kern-.7pt\fwd E3\kern-.7pt}

\def\vEE{\kern-.7pt\fwd{\E}3\kern-.7pt}
\def\vF{\kern-.7pt\fwd F3\kern-.7pt}

\def\vG{\kern-.7pt\fwd G3\kern-.7pt}
\def\vH{\kern-.5pt\fwd H3\kern-.5pt}
\def\vP{\kern-.7pt\fwd P3\kern-.6pt}

\def\specrel#1#2{\mathrel{\mathop{\kern0pt #1}\limits_{#2}}}

\title{On the homology of locally finite graphs}

\author{Reinhard Diestel and Philipp Spr\"ussel}
 \date{}

\begin{document}      

\maketitle

\begin{abstract}
We show that the topological cycle space of a locally finite graph is a canonical quotient of the first singular homology group of its Freudenthal compactification, and we characterize the graphs for which the two coincide. We construct a new singular-type homology for non-compact spaces with ends, which in dimension~1 captures precisely the topological cycle space of graphs but works in any dimension.
   \end{abstract}

\section{Introduction}

Graph homology is traditionally, and conveniently, simplicial: a graph $G$ is viewed as a 1-complex, and one considers its first simplicial homology group. In graph theory, coefficients are typically taken from a field such as~$\FF_2$, $\RR$ or~$\CC$, which makes the group into a vector space called the {\it cycle space\/} of~$G$.

For reasons to become apparent later we denote this space as $\Cfin = \Cfin(G)$. For the moment it will suffice to take our coefficients from $\FF_2$ and interpret the elements of $\Cfin$ as sets of edges.  For finite graphs~$G$, there are a number of classical theorems relating $\Cfin(G)$ to other properties of~$G$, such as planarity. (Think of MacLane's or Whitney's theorem, or the Kelmans-Tutte planarity criterion.) The cycle space $\Cfin$ has thus become one of the standard aspects of finite graphs used in their structural analysis.

When $G$ is infinite, however, the space $\Cfin$ no longer adequately describes the homology of~$G$. Most of the theorems describing the interaction of $\Cfin$ with other properties of~$G$---including all those cited above---fail when $G$ is infinite. However, the traditional role of the cycle space in these cases can be restored by defining it slightly differently: when $G$ is locally finite, one takes as generators not the edge sets of the (finite) cycles in~$G$---as one would to generate~$\Cfin$---but the (possibly infinite) edge sets of all its {\it topological circles\/}, the homeomorphic images of  the circle $S^1$ in the Freudenthal compactification $|G|$ of $G$ by its ends. (One also has to allow infinite sums in the generating process; for these to be well-defined, each edge may occur in only finitely many terms.) We denote this more general space, the {\it topological cycle space\/} of~$G$, by $\C = \C(G)$.

The space $\C$ had not been considered in graph theory before \cite{CyclesI} appeared, and it has been surprisingly successful at extending the classical cycle space theory of finite graphs to locally finite graphs; see e.g.~\cite{LocFinTutte,Duality,Partition,LocFinMacLane,AgelosFleisch,Arboricity}, or \cite{RDsBanffSurvey} for a survey. However, a question raised in~\cite{CyclesIntro} but still unanswered is how new, from a topological viewpoint, is the homology described implicitly by~$\C$. It is the purpose of this paper to clarify this relationship.

Our first result says that there is indeed a classical homology theory whose first group is isomorphic to~$\C$: the \v{C}ech homology of~$|G|$. However, the group of~$\C(G)$ as such does not carry all the information that makes it relevant to the study of the (combinatorial) structure of~$G$; one also needs to know, for example, which group elements correspond to circuits and which do not. These details are lost in the transition between $\C$ and the \v{C}ech homology, which is why we do not pursue this approach further.

Since topological circles are (images of simplices representing) singular 1-cycles in~$|G|$, it is also natural to ask how closely $\C(G)$ is related to the first singular homology group of~$|G|$. Indeed it is not clear whether the two coincide by some natural canonical isomorphism, so that $\C(G)$ would be just another way of looking at~$H_1(|G|)$.

Our first major aim in this paper is to answer this question. We begin by studying the homomorphism $f\colon H_1(|G|)\to\C(G)$ that should serve as the desired canonical isomorphism if indeed there is one. Surprisingly, this homomorphism is easily seen to be surjective. However, it turns out that it usually has a non-trivial kernel. Thus~$\C(G)$, despite looking `larger' because we allow infinite sums in its generation from elementary cycles, turns out to be a (usually proper) quotient of~$H_1(|G|)$.

For the proof that $f$ has a non-trivial kernel we have to go some way towards the solution of another problem (solved fully in~\cite{FundGp}): to find a combinatorial description of the fundamental group of the space~$|G|$ for an arbitrary connected locally finite graph~$G$.%
   \footnote{Covering space theory does not apply since, trivial exceptions aside, $|G|$~is not semi-locally simply connected at ends.}
 We describe~$\pi_1(|G|)$, as for finite~$G$, in terms of reduced words in the oriented chords of a spanning tree. However, when $G$ is infinite this does not work for arbitrary spanning trees; we have to allow infinite words of any countable order type; and reduction by cancelling adjacent inverse sequences of letters does not suffice. However, the kind of reduction we need can be described in terms of word reductions in the free groups $F_I$ on all the finite subsets $I$ of chords, which enables us to embed the group $F_\infty$ of infinite reduced words in the inverse limit of those~$F_I$, and handle it in this form. On the other hand, mapping a loop in $|G|$ to the sequence of chords it traverses, and then reducing that sequence (or word), turns out to be well defined on homotopy classes and hence defines an embedding of $\pi_1(|G|)$ as a subgroup in~$F_\infty$. This combinatorial description of $\pi_1(|G|)$ then enables us to define an invariant on 1-chains in $|G|$ that can distinguish some elements of the kernel of $f$ from boundaries of singular 2-chains, completing the proof that $f$ need not be injective.

Our second aim, then, is to begin to reconcile these different treatments of the homology of non-compact spaces between topology and graph theory. In Section~7 we present a first solution to this problem: we define a natural singular-type homology which, applied to graphs, captures precisely their topological cycle space. Essentially, we shall allow infinite sums of cycles and boundaries when building their respective groups, but start from \emph{finite} chains with zero boundary as generators. Thus, topological circles are 1-cycles, as desired. But if $G$ is a 2-way infinite path, then its edges form an infinite 1-chain with zero boundary that is \emph{not} a 1-cycle, because it is not a (possibly infinite) sum of {\em finite\/} 1-cycles. Our homology thus lies between the usual singular homology and the `open homology' that is built from arbitrary locally finite chains without any further restriction (see, eg, \cite[Ch.~8.8]{FuchsViro}).

Formally, we define our homology in Section~7 in a very general setting: all we require is a topological space in which some points are distinguished as `ends'. A~drawback of this combination of simplicity with generality is that although we can define the groups as desired, our definitions do not lead to a homology theory in the full axiomatic sense. However, it is possible to do that too: to construct a singular homology theory that does satisfy the axioms and which, for graphs, is equivalent to the homology of Section~7 and hence to the topological cycle space. We may thus view our intuitive homology of Section~7 as a stepping stone towards this more general theory, to be developed in~\cite{Hom2}, which will work for any locally compact space with ends. In both settings, ends play a role that differs crucially from that of ordinary points, which enables this homology to capture the properties of the space itself in a way similar to how the topological cycle space describes a locally finite graph.

Our hope with this paper is to stimulate further work in two directions. One is that its new topological guise makes the cycle space $\C$ accessible to topological methods that might generate some windfall for the study of graphs. And conversely, that as the approach that gave rise to~$\C$ is made accessible to more general spaces and higher dimensions, its proven usefulness for graphs might find some more general topological analogues---perhaps based on the homology theory developed in~\cite{Hom2} from the ideas presented in this paper.

\section{Terminology and basic facts}

In this section we briefly run through any non-standard terminology we use. We also list without proof a few easy lemmas that we shall need, and use freely, later on.

For graphs we use the terminology of~\cite{DiestelBook05}, for topology that of Hatcher~\cite{Hatcher}. We reserve the word `boundary' for homologousal contexts and use `frontier' for the closure of a set minus its interior. Our use of the words `path', `cycle' and `loop', where these terminologies conflict, is as follows. The word {\em path\/} is used in both senses, according to context (such as `{\em path in~$X$\/}', where $X$ was previously introduced as a graph or as a topological space). Note that while topological paths need not be injective, graph-theoretical paths are not allowed to repeat vertices or edges. The term {\em cycle\/} will be used in the topological sense only, for a (usually 1-dimensional) singular chain with zero boundary. When we do need to speak about graph-theoretic cycles (i.e., about finite connected graphs in which every vertex has exactly two incident edges) we shall instead refer to the edge sets of those graphs, which we shall call \emph{circuits}.
Our graphs may have multiple edges but no loops. This said, we shall from now on use the term {\em loop\/} topologically: for a topological path $\sigma\colon [0,1]\to X$ with $\sigma(0) = \sigma(1)$. This loop is {\em based at\/} the point~$\sigma(0)$. Given any path $\sigma\colon [0,1]\to X$, we write $\sigma^-\colon s\mapsto \sigma(1-s)$ for the inverse path. An \emph{arc} in a topological space is a subspace homeomorphic to~$[0,1]$.

\begin{lemma}[{\cite[p.~208]{ElemTop}}]\label{arc}
The image of a topological path with distinct endpoints $x,y$ in a Hausdorff space $X$ contains an arc in $X$ between $x$ and~$y$.
\end{lemma}

All homotopies between paths that we consider are relative to the first and last point of their domain, usually~$\{0,1\}$. We shall often construct homotopies between paths segment by segment. The following lemma enables us to combine certain homotopies defined separately on infinitely many segments.

\begin{lemma}\label{infhomotopies}
  Let $\alpha,\beta$ be paths in a topological space $X$. Assume that there is a sequence $(a_0,b_0),(a_1,b_1),\dotsc$ of disjoint subintervals of $[0,1]$ such that $\alpha$ and $\beta$ conincide on $[0,1]\sm\bigcup_n(a_n,b_n)$, while each segment $\alpha \restr [a_n,b_n]$ is homotopic in $\alpha([a_n,b_n])\cup\beta([a_n,b_n])$ to $\beta \restr [a_n,b_n]$. Then $\alpha$ and $\beta$ are homotopic.
\end{lemma}

\begin{proof}
  Write $D\assign\bigcup_n(a_n,b_n)$. For every $n\in\N$ let $F^n={(f^n_t)}_{t\in[0,1]}$ be a homotopy in $\alpha([a_n,b_n])\cup\beta([a_n,b_n])$ between $\alpha \restr [a_n,b_n]$ and $\beta \restr [a_n,b_n]$. We define the desired homotopy $F={(f_t)}_{t\in[0,1]}$ between $\alpha$ and $\beta$ as
  \begin{equation*}
    f_t(x)\assign
    \begin{cases}
      f^n_t(x) & \text{if }x\in(a_n,b_n),\\
      \alpha(x)=\beta(x) & \text{if }x\in[0,1]\sm D.
    \end{cases}
  \end{equation*}
  Clearly, $f_0=\alpha$ and $f_1=\beta$. It remains to prove that $F$ is continuous.

Let $x,t\in[0,1]$ and a neighbourhood $U$ of $F(x,t)$ in $X$ be given. We find an $\eps>0$ so that $F((x-\eps,x],(t-\eps,t+\eps)) \subset U$; the case $F([x,x+\eps),(t-\eps,t+\eps)) \subset U$ is analogous. Suppose first that there is an $\eps_0>0$ such that $(x-\eps_0,x)\subset D$. As the intervals $(a_i,b_i)$ are disjoint, this means that $(x-\eps,x)\subset(a_n,b_n)$ for some~$n$. Then $(x-\eps_0,x]\subset[a_n,b_n]$, and hence $F \restr (x-\eps_0,x]\times[0,1] = F^n \restr (x-\eps_0,x]\times[0,1]$. As $F^n$ is continuous, there is an $\eps<\eps_0$ with $F((x-\eps,x],(t-\eps,t+\eps)) \subset U$.

Now suppose that for every $\eps > 0$ the interval $(x-\eps,x)$ meets $[0,1]\sm D$. Then also $x\in [0,1]\sm D$, and hence $F(x,t)=\alpha(x)=\beta(x)$. Pick $\eps>0$ with $x-\eps\in[0,1]\sm D$ small enough that both $\alpha$ and $\beta$ map $[x-\eps,x]$ into $U$. Then $F((x-\eps,x],(t-\eps,x+\eps))\subset U$. Indeed, for every $x'\in(x-\eps,x]\sm D$ and every $t'\in(t-\eps,t+\eps)$ we have $F(x',t')=\alpha(x')=\beta(x')\in U$. On the other hand, for every $x'\in(x-\eps,x]\cap D$ and $t'\in(t-\eps,t+\eps)$ we have $x'\in(a_n,b_n)$ for some $n$. As $x$ and $x-\eps$ lie in $[0,1]\sm D$, we have $(a_n,b_n)\subset(x-\eps,x)$ and hence $F(x',t')=F^n(x',t')\in\alpha([a_n,b_n])\cup\beta([a_n,b_n])\subset U$.
\end{proof}

All the CW-complexes we consider will be \emph{locally finite}: every point has an open neighbourhood meeting only finitely many cells. Note that a compact subset of such a complex can meet the closures of only finitely many cells, and that locally finite CW-complexes are metrizable~\cite[Ch.~II, Prop.~3.8]{LundellWeingram} and thus first-countable.

Locally finite CW-complexes can be compactified by adding their \emph{ends}. This compactification can be defined, without reference to the complex, for any connected, locally connected, locally compact topological space $X$ with a countable basis. Very briefly, an \emph{end} of $X$ is an equivalence class of sequences $U_1\supseteq U_2\supseteq \ldots$ of connected non-empty open sets with compact frontiers and an empty overall intersection of closures, $\bigcap_n\overline U_n = \emptyset$, where two such sequences $(U_n)$ and $(V_m)$ are \emph{equivalent} if every $U_n$ contains all sufficiently late $V_m$ and vice versa. This end is said to \emph{live in} each of the sets~$U_n$, and every $U_n$ together with all the ends that live in it is \emph{open} in the space whose point set is the union of $X$ with the set $\Omega(X)$ of its ends and whose topology is generated by these open sets and those of~$X$. This is a compact space, the \emph{Freudenthal compactification} of~$X$ \cite{Freudenthal31, Freudenthal42}. More topological background on this can be found in~\cite[Ch.\ I.9]{BauesQuintero}; for applications to groups see e.g.~\cite{RoggiEndsI, RoggiEndsII, ThomassenWoess, woessBook}.

For graphs, ends and the Freudenthal compactification are more usually defined combinatorially, as follows~\cite[Ch.~8.5]{DiestelBook05},~\cite{halin64, jung71}. Let $G$ be a connected locally finite graph. A 1-way infinite path in $G$ is a \emph{ray}. Two rays are \emph{equivalent} if no finite set of vertices separates them in~$G$, and the resulting equivalence classes are the \emph{ends} of~$G$. It is not hard to see that this combinatorial definition of an end coincides with the topological one given earlier for locally finite complexes.%
   \footnote{For graphs that are not locally finite, the two concepts differ~\cite{Ends}.}
   The Freudenthal compactification of~$G$ is now denoted by~$|G|$; its topology is generated by the open sets of $G$ itself (as a 1-complex) and the sets $\hat C (S,\omega)$ defined for every end $\omega$ and every finite set $S$ of vertices, as follows. $C(S,\omega) =: C$ is the unique component of $G-S$ in which $\omega$ \emph{lives} (i.e., in which every ray of~$\omega$ has a \emph{tail}, or subray), and $\hat C (S,\omega)$ is the union of $C$ with the set of all the ends of $G$ that live in $C$ and the (finitely many) open edges between $S$ and~$C$.%
   \footnote{The definition given in~\cite{DiestelBook05} is formally more general, but equivalent to the simpler definition given here when $G$ is locally finite. Generalizations are studied in~\cite{KroenEnds, ThomassenVellaContinua}.}
Note that the frontier of $\hat C (S,\omega)$ in $|G|$ is a subset of~$S$, and that every ray converges to the end containing it. See \cite{RDsBanffSurvey} for (much) more on~$|G|$.

The end structure of $G$ is best reflected by a \emph{normal spanning tree}; such trees exist in every connected countable graph~\cite{DiestelBook05,jung69}. A spanning tree $T$ of~$G$ with root~$r$ is \emph{normal} if the vertices $u,v$ of every edge $e=uv$ of $G$ are comparable in the order $\le$ which $(T,r)$ induces on~$V(G)$. (Recall that $u\le v$ if $u$ lies on the unique $r$--$v$ path $rTv$ of $T$ between $r$ and~$v$.) A key property of normal spanning trees is that the intersection of the down-closures of two vertices separates them in~$G$. This implies that every end of $G$ is represented by a unique ray in $T$ starting at~$r$, and hence that adding all the ends of $G$ to $T$ does not create any circles. More generally, it is not hard to prove the following:

\begin{lemma}\label{clNST}
Let $T$ be a normal spanning tree of~$G$, and let $\Tbar := T\cup\Omega(G)$ denote its closure in~$|G|$. Then for every closed connected set $X\sub\Tbar$ and $x\in X$ there is a deformation retraction of $X$ onto~$x$.
\end{lemma}

\begin{proof}
   Let $X$ be a closed connected subset of~$\Tbar$, and let $x\in X$. Then $X$ is also closed in $|G|$ and hence arc-connected~\cite[Theorem~2.6]{TST}. For every $y\in X$ there is a unique $x$--$y$~arc $x\Tbar y$ in $\Tbar$~\cite[Theorem~8.5.7]{DiestelBook05} which hence lies in~$X$. The space $\Tbar$ is metrizable so that every edge between levels $n$ and $n+1$ has length $1/2^{n+1}$ and hence every end has distance~$1$ from the root~\cite{diestelESST}. $X$~inherits this metric $d$, note that $d(x,y)\le 2$ for all $y\in X$. Further, if $z\in y\Tbar y'$ for some $y,y'\in\Tbar$ we have $d(y,y')=d(y,z)+d(z,y')$. We construct a homotopy $F$ in $\Tbar$ from the identity on $\Tbar$ to the map $\Tbar \to \{x\}$; then we have $F(y,t) \in x\Tbar y \sub X$ for every $y\in X$ and $t\in[0,1]$, and hence $F \restr (X\times[0,1])$ will be the desired homotopy for $X$. For every $y\in\Tbar$ and $t\in[0,1]$ let $F(y,t)$ be the unique point on $x\Tbar y$ at distance $(1-t)d(x,y)$ from $x$.

   For the proof that $F$ is continuous, we show that $d(F(y,t),F(y',t)) \le d(y,y')$ for every $y,y'\in\Tbar$ and $t\in[0,1]$; then for every $\eps>0$ and every $y,y'\in\Tbar$ with $d(y,y')<\eps/3$ and $t,t'\in[0,1]$ with $|t-t'|<\eps/3$ we have
   \begin{align*}
     d(F(y,t),F(y',t')) &\le d(F(y,t),F(y',t)) + d(F(y',t),F(y',t'))\\
     &\le d(y,y') + |t-t'|\cdot d(x,y')\\
     &< \eps/3 + (\eps/3) \cdot 2 = \eps.
   \end{align*}
   As $x\Tbar y$ and $x\Tbar y'$ are closed, there is a last $z$ point on $x\Tbar y$ that is also in $x\Tbar y'$. As $\Tbar$ contains a unique arc between any two points in $\Tbar$, we have $y\Tbar y'=y\Tbar z \cup z\Tbar y'$ and hence $d(y,y')=d(y,z)+d(z,y')$. If $F(y,t)\in z\Tbar y$ and $F(y',t)\in z\Tbar y'$, then $d(F(y,t),F(y',t)) \le d(F(y,t),z) + d(z,F(y',t)) \le d(y,z)+d(z,y') = d(y,y')$. Otherwise both $F(y,t)$ and $F(y',t)$ are contained in $x\Tbar y$ or in $x\Tbar y'$. In particular, one of $F(y,t),F(y',t)$ lies on the arc between the other and~$x$. Then $d(F(y,t),F(y',t)) = |d(x,F(y,t))-d(x,F(y',t))| = (1-t)\cdot |d(x,y)-d(x,y')| \le d(y,y')$.
\end{proof}

Lemma~\ref{clNST} implies that $\Tbar$ contains no topological circle. Equivalently: for any two points $x,y\in\Tbar$ there is a unique arc in $\Tbar$ between $x$ and~$y$. We denote this arc by~$x\Tbar y$. The uniqueness of $x\Tbar y$ implies that none of its inner points can be an end. (Every arc containing an end also contains a vertex, and any two vertices of $T$ can also be joined by an arc in $T$ itself.)

When $T$ is a normal spanning tree of~$G$, every end $\omega$ in $|G|$ has a neighbourhood basis consisting of open sets $\hat C = \hat C(S,\omega)$ such that $S$ is closed downwards, i.e.\ where $s'\le s\in S$ implies $s'\in S$. We call these sets $\hat C$ the \emph{basic open neighbourhoods} of the ends%
   \footnote{The \emph{basic open neighbourhoods} of a point $x\in G$ are the connected open neighbourhoods of $x$ containing no vertex other than possibly~$x$.}
of~$G$ (given~$T$). An important property of these sets is that for any two points $x,y\in \hat C$ we also have $x\Tbar y\sub\hat C$.

Now let $S'= S\cup N(S)$, the (finite) set of vertices in $S$ and their neighbours. We call the subset $\hat C(S',\omega)$ of $C(S,\omega)$ the \emph{inside of~$\hat C(S,\omega)$ around~$\omega$}. Note that the neighbours $v$ of vertices $u\in C(S',\omega)$, as well as the edges~$uv$, also lie in~$C(S,\omega)$.

More background on normal spanning trees, including an existence proof, can be found in~\cite[Ch.~8]{DiestelBook05},~\cite{DiestelLeaderBGC, DiestelLeaderNST}.

\medbreak

Let us now introduce the topological cycle space $\C$ of~$G$. This is usually defined over~$\FF_2$ (which suffices for its role in graph theory), but we wish to prove our main results more generally with integer coefficients. (The $\FF_2$ case will follow, but it should be clear right away that the non-injectivity of our homomorphism $H_1\to\C$ is not just a consequence of a wrong choice of coefficients.) We therefore need to speak about orientations of edges.

An edge $e=uv$ of $G$ has two {\em directions\/}, $(u,v)$ and~$(v,u)$. A~triple $(e,u,v)$ consisting of an edge together with one of its two directions is an {\em oriented edge\/}. The two oriented edges corresponding to $e$ are its two {\em orientations\/}, denoted by $\ve$ and~$\ev$. Thus, $\{\ve,\ev\} = \{(e,u,v), (e,v,u)\}$, but we cannot generally say which is which. However, from the definition of $G$ as a CW-complex we have a fixed homeomorphism $\theta_e\colon [0,1]\to e$. We call $(\theta_e(0),\theta_e(1))$ the {\em natural direction\/} of~$e$, and $(e,\theta_e(0),\theta_e(1))$ its {\em natural orientation\/}.

Given a set $E$ of edges in~$G$, we write $\vE$ for the set of their orientations, two for every edge in~$E$. Given a partition $(U,V)$ of the vertex set of~$G$, we write $\vE(U,V)$ for the set of all its oriented edges $(e,u,v)$ with $u\in U$ and $v\in V$, and call this set an {\em oriented cut\/} of~$G$.

Let $\sigma\colon [0,1]\to |G|$ be a path in~$|G|$. Given an edge $e = uv$ of~$G$, if $[s,t]$ is a subinterval of $[0,1]$ such that $\{\sigma(s),\sigma(t)\} = \{u,v\}$ and $\sigma((s,t)) = \interior e$, we say that $\sigma$ \emph{traverses~$e$} on~$[s,t]$. It does so \emph{in the direction of~$(\sigma(s), \sigma(t))$}, or \emph{traverses $\ve = (e,\sigma(s),\sigma(t))$}. We then call its restriction to $[s,t]$ a \emph{pass of $\sigma$ through~$e$}, or~$\ve$, \emph{from $\sigma(s)$ to~$\sigma(t)$}.

Using that $[0,1]$ is compact and $|G|$ is Hausdorff, one easily shows that a path in $|G|$ contains at most finitely many passes through any given edge:

\begin{lemma}\label{pass}
A path in $|G|$ traverses each edge only finitely often.
\end{lemma}

\begin{proof}
Let $\sigma$ be a path in $|G|$, and let $e=uv$ be an edge such that $\sigma$ contains infinitely many passes $\sigma\restr [s_n,t_n]$ through~$e$ ($n=1,2,\dots$). As $[0,1]$ is compact, the sequence $s_1,s_2,\dotsc$ has an accumulation point~$x$, which is also an accumulation point of $t_1,t_2,\dotsc$. But now $\sigma$ fails to be continuous at~$x$, because $\{\sigma(s_n), \sigma(t_n)\} = \{u,v\}$ for each~$n$ but each of $u$ and $v$ has a neighbourhood not containing the other.
\end{proof}

A loop that is injective on~$[0,1)$ is a {\em circle\/} in~$|G|$. (In most of our references, the term {\em circle\/} is used for the image of such a loop.) The set of all edges traversed by a circle is a {\em circuit\/}. It is easy to show that the image of a circle is uniquely determined by its circuit~$C$, being the closure of $\bigcup C$ in~$|G|$.

Let $\vEE = \vEE(G)$ denote the set of all integer-valued functions $\phi$ on the set $\vE$ of all oriented edges of $G$ that satisfy $\phi(\ev) = -\phi(\ve)$ for all $\ve\in \vE$. This is an abelian group under pointwise addition. A~family $(\phi_i\mid i\in I)$ of elements of $\vEE$ is {\em thin\/} if for every $\ve\in\vE$ we have $\phi_i(\ve)\ne 0$ for only finitely many~$i$. Then $\phi = \sum_{i\in I} \phi_i$ is a well-defined element of~$\vEE$: it maps each $\ve\in\vE$ to the (finite) sum of those $\phi_i(\ve)$ that are non-zero. We shall call a function $\phi\in\vEE$ obtained in this way a {\em thin sum\/} of those~$\phi_i$.

We can now define our oriented version of the topological cycle space of~$G$. When $\alpha$ is a circle in~$|G|$, we call the function $\phi_\alpha\colon \vE\to\Z$ defined by
   $$\phi_\alpha\colon \ve\mapsto \left\{\hskip-6pt\begin{array}{rl}
1 & \textrm{if $\alpha$ traverses~$\ve$}\\
-1 & \textrm{if $\alpha$ traverses~$\ev$}\\
0 & \textrm{otherwise.}
\end{array} \right.$$
   an {\em oriented circuit\/} in~$G$, and write $\vCC = \vCC(G)$ for the subgroup of $\vEE$ formed by all thin sums of oriented circuits.
   
We remark that $\vCC$ is closed also under infinite thin sums~\cite[Cor.~5.2]{CyclesI}, but this is neither obvious nor generally true for thin spans of subsets of~$\vEE$~\cite[Sec.~3]{Basis}. We remark further that composing the functions in $\vCC$ with the canonical homomorphism $\Z\to\Z_2$ yields the usual {\em topological cycle space\/} $\C(G)$ of~$G$ as studied in~\cite{LocFinTutte, Duality,Partition, LocFinMacLane, Degree,CyclesI, CyclesII,TST, AgelosFleisch, Geo, Arboricity}, the $\FF_2$ vector space of subsets of $E$ obtained as thin sums of (unoriented) circuits.
   
The topological cycle space $\C(G)$ can be characterized as the set of those subsets of $E$ that meet every {\em finite\/} cut of $G$ in an even number of edges~\cite[Thm.~7.1]{CyclesI},~\cite[Thm.~8.5.8]{DiestelBook05}. The characterization has an oriented analogue:

\begin{theorem}\label{orthogonal}
   An element $\phi$ of $\vEE$ lies in $\vCC$ if and only if $\sum_{\ve\in\vF} \phi(\ve) = 0$ for every finite oriented cut $\vF$ of~$G$.
   \end{theorem}

The proof of Theorem~\ref{orthogonal} is not completely trivial. But it adapts readily from the unoriented proof given e.g.\ in~\cite{DiestelBook05}, which we leave to the reader to check if desired.

When $G$ is fixed, we write $C_n$ for the group of singular $n$-chains in~$|G|$ (with coefficients in $\Z$ unless otherwise mentioned), $Z_n = \ke\partial_n$ and $B_n = \im\partial_{n+1}$ for the corresponding groups of cycles and boundaries, and $H_n = Z_n/B_n$. We view all singular 1-simplices as maps from the real interval $[0,1]$ to~$|G|$. The homology class of a cycle~$z$ is denoted by~$[z]$.
   
A cycle that can be written as a sum of 1-simplices no two of which share their first point is an {\em elementary cycle\/}. Every 1-cycle is easily seen to be a sum of elementary 1-cycles, a decomposition which is not normally unique.

The following lemma enables us to subdivide or concatenate the simplices in a 1-cycle while keeping it in its homology class.

\begin{lemma}\label{subdivide}
   Let $\sigma$ be a singular 1-simplex in~$|G|$, and let $s\in (0,1)$. Write $\sigma'$ and $\sigma''$ for the 1-simplices obtained from the restrictions of $\sigma$ to $[0,s]$ and to $[s,1]$ by reparametrizing linearly. Then $\sigma'+\sigma''-\sigma\in B_1$.
   \noproof
   \end{lemma}

When $\sigma$ is a summand in a cycle $z\in Z_1$, we shall say that the equivalent cycle $z'$ obtained by replacing $\sigma$ with $\sigma'+\sigma''$ in the sum arises by {\em subdividing~$\sigma$} (at~$s$ or at~$\sigma(s)$). A~frequent application of Lemma~\ref{subdivide} is the following:

\begin{corollary}\label{loops}
Every non-zero element of $H_1(|G|)$ is represented by a sum of loops each based at a vertex.
  \end{corollary}

\begin{proof}
Pick a cycle representing a given homology class, and decompose it into elementary cycles. Use Lemma~\ref{subdivide} to concatenate their simplices into a single loop. If such a loop~$\alpha$ passes through a vertex, we can subdivide it there and suppress its original boundary point, obtaining a homologous loop based at that vertex. If $\alpha$ does not pass through a vertex, then ${\im\alpha \sub\interior e}$ for some edge~$e$ (since non-trivial sets of ends are never connected), so $\alpha$ is null-homotopic and $[\alpha] = 0$.
   \end{proof}

\section{\boldmath $\vCC(G)$ and the \v{C}ech homology}\label{Cech}

In this section we briefly describe the relationship between the topological cycle space of a graph with ends and its \v{C}ech homology. We shall see that their groups are canonically isomorphic, but also that this isomorphism is not enough to capture the relevance of $\vC(G)$ to the structure of~$G$---the reason why graph theorists study cycle spaces in the first place. The material from this section will not be needed in the rest of the paper.

The \v{C}ech homology of a space is an alternative to singular homology for spaces that do not have a simplicial homology, and we begin by recalling its definition. Consider a space $X$ and an open cover $\U$ of $X$. Then $\U$ defines a simplicial complex $X_{\U}$, the \emph{nerve} of $\U$: The $0$-simplices of $X_{\U}$ are the elements of $\U$, and any $n+1$ elements of $\U$ form an $n$-simplex if and only if they have a nonempty intersection. For two open covers $\U,\U'$ let $\U\le\U'$ if $\U'$ is a refinement of $\U$. In this case, it is easy to define a continuous map from $X_{\U'}$ to $X_{\U}$: For each $0$-simplex $U$ of $X_{\U'}$ (i.e.\ $U\in\U'$) there is a $0$-simplex $\pi(U)$ of $X_{\U}$ (an element of $\U$) that contains it. Map each $U$ to $\pi(U)$ and extend this map linearly to the higher-dimensional simplices in $X_{\U'}$ so as to obtain a map $\rho:X_{\U'}\to X_{\U}$. As $U\in\U'$ can be contained in more than one element of $\U$, the choice of $\pi:\U'\to\U$ is not unique and neither is $\rho$. But it is easy to see that all possible choices of $\pi$ induce homotopic maps $\rho$ and hence the induce a unique homomorphism $\rho_{\U'}^{\U}:H_n(X_{\U'})\to H_n(X_{\U})$ on homology. Now the homology groups $H_n(X_{\U})$ for all open covers $\U$ together with the homomorphisms $\rho_{\U'}^{\U}$ form an inverse family. Define the \emph{$n$th \v{C}ech homology group} $\check H_n(X)$ to be the inverse limit of the $H_n(X_{\U})$.

For locally finite graphs the first \v{C}ech homology group and the topological cycle space coincide:

\begin{theorem}\label{CechequalsC}
  For a locally finite graph $G$ we have a canonical isomorphism $\check H_1(|G|) \simeq \vCC(G)$.
\end{theorem}

\begin{proof}
  To compute the inverse limit of the groups $H_1(X_{\U})$ it suffices to to consider a family $\mathfrak{U}$ of open covers of $|G|$ that contains a refinement for every open cover of $|G|$, and to compute the inverse limit of the inverse family $(H_1(|G|_{\U}))_{\U\in\mathfrak{U}}$. We will now construct a suitable $\mathfrak{U}$.
  
  Let $T$ be a normal spanning tree of $G$ and denote the subtree induced by the first $n$ levels by $T_n$. Now for each $n$ and each $\eps>0$ let $\mathfrak{U}$ contain an open cover $\U_{n,\eps}$ consisting of the following sets: An open star of radius $\eps$ around each vertex $v\in V(T_n)$, finitely many open subintervals of length $\eps$ of each edge $e\in E(T_n)$,
   \COMMENT{Of course, the subintervals are supposed to cover the whole edge.}
  and the sets $\hat C(V(T_n),\omega)$ for each end of omega. Note that $\U_n$ is a finite family as $G-V(T_n)$ has only finitely many components.
  
  Using that $|G|$ is compact, it is not hard to see that for each open cover $\U$ of $|G|$ some $\U_{n,\eps}$ is a refinement of $\U$.
   \COMMENT{For every end $\omega$ of $G$ choose a nbhd $\hat C(S_{\omega},\omega)$ contained in some $U\in\U$. Since $|G|$ is compact, finitely many of these nbhds suffice to cover all of $\Omega(G)$. Let $S$ be the union of those finitely many $S_{\omega}$, then each set $\hat C(S,\omega)$ is contained in some $U\in\U$. Choose $n$ large enough so that $S\subset V(T_n)$. As $G[V(T_n)]$ is compact, there is an $\eps>0$ such that for every point in $G[V(T_n)]$ the $\eps$-nbhd around it is contained in some set in $\U$. Then $\U_{n,\eps}$ is a refinement of $\U$.
   
   NOTE: It is NOT enough to consider only the covers $\U_{n,\eps_n}$ for fixed $\eps_n$ with $\lim_{n\to\infty}\eps_n=0$. If an open cover consists of open $\frac12\eps_n$-balls around every vertex in the $n$th level of $T$, plus all open edges and some open sets covering the ends, then this cover will not have a refinement among the $\U_{n,\eps_n}$.}
  Clearly, every $|G|_{\U_{n,\eps}}$ retracts to the graph $G_n$ obtained from $G$ by contracting all components of $G-T_n$, and hence the homology group $H_1(|G|_{\U_{n,\eps}})=H_1(G_n)$ is a direct product of $\Z$'s, one for each chord of $T$ with at least one endvertex in $T_n$. Thus $\check H_1(|G|)$ also is the direct product of copies of $\Z$, one for each chord of $T$. As the same is true for $\vCC(G)$, we have that $\check H_1(|G|)$ and $\vCC(G)$ are canonically isomorphic.
\end{proof}

Although the first \v{C}ech homology group is isomorphic to the group of the topological cycle space, it does not sufficiently reflect the combinatorial properties of $\vCC(G)$. For example, a number of classical results about the cycle space say which circuits generate it---as do the non-separating chordless circuits in a $3$-connected graph, say. In the \v{C}ech homology, however, it is not possible to decide whether a homology class in $\check H_1(|G|)$ corresponds to a circuit in $G$. One might think that since a homology class $c\in\check H_1(|G|)$ corresponds to a family $(c_{n,\eps})$ of homology classes in the groups $H_1(|G|_{\U_{n,\eps}}) = H_1(G_n)$, the class $c$ should correspond to a circuit if every $c_{n,\eps}$ with sufficiently large $n$ corresponds to a circuit in $G_n$. But this is not the case: the limit of a sequence of cycle space elements in the $G_n$ can be a circuit even if the elements of the sequence are not circuits in the~$G_n$.
   \COMMENT{Example: (One might want to draw a picture whilst reading this comment...) Start with a "wide ladder" with three stiles $x_1^1,x_2^1,\dotsc$, $x_1^2,x_2^2,\dotsc$, and $x_1^3,x_2^3,\dotsc$, then attach infinitely many (oridinary) ladders by identifying the first rung of the $n$th ladder $L_n$ with the edge $x_{2n-1}^1x_{2n}^1$ to obtain the graph $G$. A normal spanning tree of $G$ is the following: Let $x_1^1$ be the root of $T$, go zigzag across the wide ladder (i.e.\ $x_1^1,x_1^2,x_1^3,x_2^3,x_2^2,x_2^1,x_3^1,\dotsc$) and zigzag across each $L_n$, starting from $x_{2n}^1$. Let $C$ be the circle consisting of the ray $x_1^1x_1^2x_1^3x_2^3x_3^3x_4^3\dotsm$, the two outer rays of each $L_n$, and the path $x_{2n}^1x_{2n}^2x_{2n+1}^2x_{2n+1}^1$ for each $n$. Then for every $m=3k-1$, $C$ does not induce a circuit on $G_m$, although it is the limits of its restrictions.}

In order to have a homology that reflects the properties of $\vCC(G)$, we thus need to take a singular approach.

\section{\boldmath Comparing $H_1(|G|)$ with $\vCC(G)$}\label{comp}

Let $G = (V,E)$ be a connected locally finite graph. Our aim in this section is to compare the first singular homology group $H_1$ of~$|G|$ (with integer coefficients) with the oriented topological cycle space $\vCC(G)$ of~$G$, the group of thin sums of oriented circuits in~$G$. When $G$ is finite then $|G| = G$, and all circuits and their thin sums are finite. Hence in this case $\vCC$ is just the first simplicial homology group of~$G$, so the two groups are indeed the same.

When $G$ is infinite, however, both circuits and thin sums can be infinite too. So they are not just the simplicial 1-cycles in~$G$. But there is an obvious singular 1-cycle in $|G|$ associated with an oriented circuit~$\phi_\alpha$: the circle~$\alpha$, viewed as a singleton 1-chain. Our aim is to extend this correspondence to one between $H_1$ and~$\vCC$.

Our approach will be to define a homomorphism $f\colon H_1\to\vEE(G)$ that counts for a given homology class $h$ how often the 1-simplices of a cycle representing~$h$, when properly concatenated, traverse a given edge~$\ve$; we then let $f(h)\in\vEE(G)$ map $\ve$ to this number.\footnote{The precise definition of $f$ will be given shortly.} We shall prove that $f(h)$ always lies in~$\vCC$ and, perhaps surprisingly, that $f$ maps $H_1$ onto~$\vCC$. However, we find that $f$ is not normally injective. Our first main result characterizes the graphs for which it is:

\begin{theorem}\label{compthm}
The map $f\colon H_1(|G|)\to\vEE(G)$ is a group homomorphism onto~$\vCC(G)$, which has a non-trivial kernel if and only if $G$ contains infinitely many (finite) circuits.
\end{theorem}

Thus, $\vCC$~turns out to be a canonical---but usually non-trivial---quotient of~$H_1$. Taking this result mod~2 answers our original question: the topological cycle space $\C$ of~$G$ is a canonical---but usually non-trivial---quotient of the singular homology group of~$|G|$ with $\FF_2$ coefficients.

We remark that the last condition in Theorem~\ref{compthm} can be rephrased in various natural ways: that $G$ has a spanning tree with infinitely many chords; that every spanning tree of $G$ has infinitely many chords; or that $G$ contains infinitely many \emph{disjoint} (finite) circuits~\cite[Ex.~37, Ch.~8]{DiestelBook05}. The remainder of this section and the next two sections will be devoted to the proof of Theorem~\ref{compthm}.

\medskip

Let us define $f$ formally. Let $S^1$ denote the unit circle in the complex plane. The elements of $H_1(S^1)$ are represented by the loops $\eta_k\colon [0,1]\to S^1$, $s\mapsto e^{2\pi iks}$, $k\in\Z$. Write $\pi\colon H_1(S^1)\to\Z$ for the group isomorphism $[\eta_k]\mapsto k$. For every edge $e$ of~$G$, let $f_e\colon |G|\to S^1$ wrap $e$ round $S^1$ in its natural direction, defining $f_e \restr\interior e$ as $\eta_1\circ\theta_e^{-1}$ and putting $f_e (|G|\sm \interior e) := 1\in\CC$. Note that $f_e$ is continuous. 

The following lemma is easy to prove using homotopies in~$S^1$, combined by Lemma~\ref{infhomotopies}:

\begin{lemma}\label{traverse}
   Let $\alpha\colon [0,1]\to |G|$ be a loop based at a vertex. If $\alpha$ traverses $e$ exactly $k$ times in its natural direction and exactly $\ell$ times in the opposite direction, then $\pi([f_e\circ\alpha]) = k-\ell$.
   \end{lemma}

\begin{proof}
   Composing a pass of $\alpha$ through~$e$ (in its natural direction) with $f_e$ yields a map from a subinterval of $[0,1]$ to $S^1$ which, after reparametrization, is homotopic to~$\eta_1$.
   
The domains of distinct passes of $\alpha$ through $e$ are closed subintervals of $[0,1]$ meeting at most in their boundary points. The rest of $[0,1]$ is a finite disjoint union of open intervals $(s,t)$ (or $(s,1]$ or $[0,t)$). Each of these is in turn a disjoint union, possibly infinite, of open intervals $(s',t')$ which $\alpha$ maps to~$\interior e$ and closed intervals which $f_e\circ\alpha$ maps to $1\in\CC$. Since $[s,t]$, by definition, contains no pass through~$e$, $\alpha$~always maps $s'$ and $t'$ to the same endvertex of~$e$. Then $\alpha\restr [s',t']$ is homotopic to the constant map to that vertex, and $(f_e\circ\alpha)\restr [s',t']$ is homotopic to the constant map to~1. These homotopies combine to a homotopy of $(f_e\circ\alpha)\restr [s,t]$ to the constant map with value~1.

We deduce that $f_e\circ\alpha$ is homotopic to a concatenation $\sigma_1\cdot\ldots\cdot\sigma_n$ of loops in $S^1$ of which (after reparametrization) $k$ are equal to~$\eta_1$ and $\ell$ are equal to the inverse loop~$\overline{\eta_1}\colon \lambda\mapsto 1-\eta_1(\lambda)$, and the rest are constant with value~1. The result follows.
\end{proof}

Given $h\in H_1 (|G|)$, we now let $f(h)\in \vEE(G)$ assign $(\pi\circ (f_e)_*) (h)\in\Z$ to the natural orientation $\ve$ of~$e$:
   $$f(h)\colon \ve\mapsto (\pi\circ (f_e)_*) (h)\in\Z\,.$$
   This completes the definition of $f\colon H_1(|G|)\to\vEE$, which is clearly a group homomorphism.

\goodbreak

\begin{lemma}\label{rangeC}
  $\im f\sub\vCC(G)$.
  \end{lemma}

\begin{proof}
   By Theorem~\ref{orthogonal} it suffices to show that for every finite oriented cut $\vF$ of $G$ and every $h\in H_1 (|G|)$ we have $\sum_{\ve\in\vF} f(h)(\ve) = 0$. Let $\vF = \vE(U,U')$ and $h$ be given, let $F = \{e\mid\ve\in\vF\}$, and assume for simplicity that the orientations $\ve\in\vF$ of these edges are their natural orientations. Since $f$ is a homomorphism, we may assume that $h$ is represented by an elementary 1-cycle, which we may choose by Corollary~\ref{loops} to consist of a loop~$\alpha$ based at a vertex. We shall prove that $\alpha$ traverses the edges in $F$ as often from $U$ to $U'$ as it does from $U'$ to~$U$. Then
 $$\sum_{\ve\in\vF} f(h)(\ve) =
   \sum_{e\in F} (\pi\circ (f_e)_*) ([\alpha]) =
   \sum_{e\in F} \pi ([f_e\circ\alpha]) = 0$$
 by Lemma~\ref{traverse}.
   
Let $[s_1,t_1],\dots,[s_n,t_n]$ be the domains of the passes of $\alpha$ through edges of~$F$. In order to prove that as many of these passes are from a vertex in $U$ to one in $U'$ as vice versa, it suffices to show that each of the segments $\beta = \alpha\restr [t_i,s_{i+1}]$ has either all its vertices in $U$ or all its vertices in~$U'$, assuming for simplicity that $\alpha$ is based at $t_n =: t_0$. If the starting vertex $\beta(t_i)$ of $\beta$ lies in~$U$, say, put
 $$s := \sup\,\{\,r\in [t_i,s_{i+1}] : V\cap \beta ([t_i,r]) \sub U\,\}\,.$$
We wish to show that $s = s_{i+1}$. If not, then $\beta(s)$ is an end, and this end lies both in the closure of $U$ and in the closure of~$U'$. But these closures are disjoint: the set $S$ of vertices incident with an edge in~$F$ is finite, and since $S$ separates $U$ from~$U'$, the neighbourhood $\hat C(S,\omega)$ of any end $\omega$ avoids either $U$ or~$U'$.
   \end{proof}

Next, we prove that $f$ is surjective. At first glance, this may seem surprising: after all, we have to capture arbitrary thin sums of oriented circuits, which may well be disjoint, by finite 1-cycles.

\begin{lemma}\label{surjectivity}
 $\im f \supseteq \vCC(G)$.
   \end{lemma}
   
\begin{proof}
   Let $\phi = \sum_{\alpha\in A} \phi_\alpha \in \vCC(G)$ be an arbitrary thin sum of oriented circuits, where each $\alpha$ is a circle in~$|G|$. Ignoring any $\phi_\alpha$ that are constant with value~0, we may assume that each $\alpha$ is based at a vertex~$v(\alpha)$. (Recall that if the image of~$\alpha$ contains no vertex it must lie inside an edge, because non-trivial sets of ends cannot be connected.) We shall construct a loop $\tau$ in $|G|$ such that $f([\tau]) = \phi$.

Let $T$ be a spanning tree of~$G$ and pick a root $r\in V$. Write $V_n$ for the set of vertices at distance~$n$ in~$T$ from~$r$, and let $T_n$ be the subtree of $T$ induced by ${V_0\cup\dots\cup V_n}$. Our first aim will be to construct a loop $\sigma$ in $|G|$ that traverses every edge of $T$ once in each direction and avoids all other edges of~$G$. We shall obtain $\sigma$ as a limit of similar loops $\sigma_n$ in $T_n\sub |G|$. We shall then incorporate our loops $\alpha\in A$ into~$\sigma$, to obtain~$\tau$. When we describe these maps informally, we shall think of $[0,1]$ as measuring time, and of a loop as a journey through~$|G|$.

Let $\sigma_0$ be the unique (constant) map $[0,1]\to T_0$. Assume inductively that $\sigma_n\colon [0,1]\to T_n$ is a loop traversing every edge of $T_n$ exactly once in each direction. Assume further that $\sigma_n$ pauses every time it visits a vertex, remaining stationary at that vertex for some time. More precisely, we assume for every vertex $v\in T_n - r$ that $\sigma_n^{-1}(v)$ is a disjoint union of as many non-trivial closed intervals as $v$ has incident edges in~$T_n$, and of one more such interval in the case of $v=r$. Let us call the restriction of $\sigma_n$ to such an interval a {\em pass\/} of $\sigma_n$ through~$v$. We are thus assuming that $\sigma_n$ is the union of its passes through the vertices and edges of~$T_n$.

Let $\sigma_{n+1}$ be obtained from $\sigma_n$ by replacing, for each leaf $v$ of~$T_n$, the unique pass of $\sigma_n$ through $v$ by a topological path that starts out remaining stationary at $v$ for some time, then visits all the neighbours of $v$ in $V_{n+1}$ in turn, pausing at each and shuttling back and forth between $v$ and those neighbours, and finally returns to $v$ to pause there. Outside the passes of $\sigma_n$ through leaves of~$T_n$, let $\sigma_{n+1}$ agree with~$\sigma_n$. Note that $\sigma_{n+1}$ satisfies our inductive assumptions for~$n+1$: it traverses every edge of $T_{n+1}$ exactly once each way, pauses every time it visits a vertex, and is the union of its passes through the vertices and edges of~$T_{n+1}$.

Let us now define~$\sigma$. Let $s\in [0,1]$ be given. If the values $\sigma_n (s)$ coincide for all large enough~$n$, let $\sigma(s) := \sigma_n (s)$ for these~$n$. If not, then $s_n := \sigma_n (s)\in V_n$ for every~$n$, and $s_0 s_1 s_2\dots$ is a ray in~$T$; let $\sigma$ map $s$ to the end of $G$ containing that ray.

Clearly every $\sigma_n$ is continuous, and $\sigma$ is continuous at points not mapped to ends. To show that $\sigma$ is continuous at every point $s$ mapped to an end $\omega = \sigma(s)$, let a neighbourhood $\hat C (S,\omega)$ of $\omega$ in $|G|$ be given. Put $s_n := \sigma_n (s)$. Choose $n$ large enough that the tree $T'$ spanned in $T$ by the vertices above~$s_n$---those vertices $v$ for which the $r$--$v$ path in $T$ contains~$s_n$---avoids the finite set~$S$. We claim that $\sigma$ maps the interval $I := \sigma_n^{-1} (s_n)$ to~$\hat C (S,\omega)$. Since $\sigma_{n+1}$ agrees with $\sigma_n$ on the boundary points of $I$ but not on~$s$, we know that $I$ is a neighbourhood of~$s$ in~$[0,1]$, so this will complete the proof that $\sigma$ is continuous. Let $t\in I$ be given. Induction on $m$ shows that $\sigma_m (I)\sub T'$ for every $m\ge n$. Hence if $\sigma(t)$ is not an end, then $\sigma(t) = \sigma_m (t) \in T'\sub \hat C (S,\omega)$ for some~$m$. But if $\sigma(t)$ is an end, then this is the end $\omega'$ of a ray that starts at $\sigma_n (t) = s_n$ and lies in~$T' \sub \hat C (S,\omega)$. Hence so does $\omega' = \sigma(t)$.
   
Let $\tau$ be obtained from $\sigma$ by replacing for every vertex $v$ one of the passes of $\sigma$ through $v$ with a concatenation of all the circles $\alpha$ with $\alpha\in A$ and $v(\alpha) = v$. Note that these are finitely many for each~$v$, because $G$ has only finitely many edges at~$v$ and $\sum_{\alpha\in A} \phi_\alpha$ is a thin sum.

Let us prove that $\tau$ is continuous. As before, this is clear at points $x\in [0,1]$ which $\sigma$ does not map to an end: for such $x$ the map $\tau$ agrees on suitable intervals $[s,x]$ and $[x,t]$ with $\sigma$ or some $\alpha\in A$, which we know to be continuous. The proof that $\tau$ is continuous at points $x$ which $\sigma$ maps to ends is similar to our earlier continuity proof for~$\sigma$. The only difference now is that we have to choose $n$ large enough also to ensure that none of the $\alpha$ with $v(\alpha)\in T'$ passes through a vertex of~$S$. Such a choice of $n$ is possible, because only finitly many edges are incident with vertices in~$S$ and the $\phi_\alpha$ form a thin family of functions. Then $\hat C (S,\omega)$ contains not only $T'$ but also the images of all $\alpha$ with $v(\alpha)\in T'$, because $\im\alpha$ is connected but does not meet the frontier~$S$ of $\hat C (S,\omega)$.

Finally, recall that $\sigma$~traverses every edge of $T$ once in each direction, and that it does not traverse any other edges. Therefore $f([\sigma]) = 0 \in \vCC(G)$, and hence $f([\tau]) = \sum_{\alpha\in A} \phi_\alpha = \phi$ as desired.
   \end{proof}

To complete the proof of Theorem~\ref{compthm} it remains to show that $f$ has a non-trivial kernel if and only if $G$ contains infinitely many circuits. The forward implication of this is easy. Indeed, suppose that $G$ contains only finitely many circuits, and let $T$ be a normal spanning tree of~$G$. Then $T$~has only finitely many chords, so $|G|$~is homotopy equivalent to a finite graph (Lemma~\ref{clNST}). Hence, as is well known, $H_1(|G|)$~equals the first simplicial homology group of $G$ viewed as a 1-complex, which in turn is clearly isomorphic to~$\vCC(G)$. Therefore $f$ must be injective.

The converse implication, surprisingly, is quite a bit harder. Assuming that $G$ contains infinitely many circuits, we shall define a loop $\rho$ in $|G|$ that traverses every edge equally often in both directions (so that $f([\rho])=0$), and which is easily seen not to be null-homotopic. To prove that $[\rho]\ne0$, however, i.e.\ that $\rho$ is not a boundary, will be harder: it turns out that we first have to understand the fundamental group of~$|G|$ a little better. With this knowledge we shall then be able to define an invariant of 1-chains that can distinguish $\rho$ from boundaries.

The next section, therefore, contains an interlude in which we describe the fundamental group $\pi_1(|G|)$ of an arbitrary locally finite connected graph~$G$ combinatorially. We shall then complete the proof of Theorem~\ref{compthm} in Section~\ref{windup}.

\section{\boldmath A combinatorial characterization of~$\pi_1(|G|)$}\label{pi1sec}

Our aim in this section is to prove some aspects of a combinatorial description of~$\pi_1(|G|)$ that we need for our proof of Theorem~\ref{compthm}. In~\cite{FundGp}, we give a more comprehensive such description; see Theorem~\ref{pi1thm} below.

When $G$ is finite, $\pi_1(|G|)$~is the free group $F$ on the set of \emph{chords} (arbitrarily oriented) of any fixed spanning tree, the edges of $G$ that are not edges of the tree. The standard description of $F$ is given in terms of reduced words of those oriented chords. The map assigning to a path in $|G|$ the sequence of chords it traverses defines the canonical group isomorphism between $\pi_1(|G|)$ and~$F$.

Our description of $\pi_1(|G|)$ for infinite~$G$ will be similar in spirit, but more complex. We shall start not with an arbitrary spanning tree but with a normal spanning tree. (The trees that work are precisely the {\em topological spanning trees\/} defined in \cite{TST} or~\cite[Ch. 8.5]{DiestelBook05}, which include the normal spanning trees.) Then every path in $|G|$ defines as its `trace' an infinite word in the oriented chords of that tree, as before. However, these words can have any countable order type, and it is no longer clear how to define reductions of words in a way that captures homotopy of paths.

Consider the following example. Let $G$ be the infinite ladder, with a spanning tree~$T$ consisting of one side of the ladder and all its rungs (drawn bold in Figure~\ref{fig:singleladder}). The path running along the bottom side of the ladder and back is a null-homotopic loop. Since it traces the chords $\ve_0, \ve_1,\dots$ all the way to~$\omega$ and then returns the same way, the infinite word $\ve_0\ve_1\dots\ev_1\ev_0$ should be reducible. But it contains no cancelling pair of letters, such as $\ve_i\ev_i$ or~$\ev_i\ve_i$.

\begin{figure}[htbp]
\centering
\noindent
\includegraphics{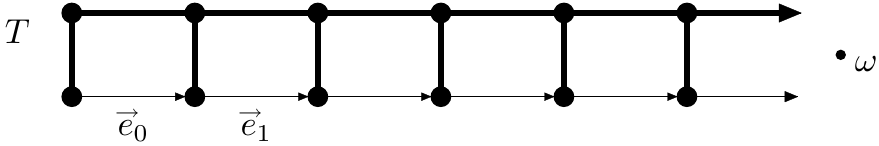}
\caption{The null-homotopic loop $\ve_0 \ve_1\dots\omega\dots\ev_1\ev_0$} 
\label{fig:singleladder}
\end{figure}%

\penalty-100

This simple example suggests that some transfinite equivalent of cancelling pairs of letters, such as cancelling inverse pairs of infinite sequences of letters, might lead to a suitable notion of reduction. However, one can construct graphs which, for any suitable spanning tree, contain null-homotopic loops whose trace of chords contains no such cancelling subsequences (of any order type).%
   \footnote{For example, consider for the binary tree~$T_2$ the loop $\sigma$ constructed in the proof of Lemma~\ref{surjectivity}, which traverses every edge of $T_2$ exactly once in each direction. This loop is null-homotopic in~$|T_2|$ (Lemma~\ref{clNST}), but no sequence of edges, of any order type, is followed immediately by the inverse of that sequence. The edges of $T_2$ aren't chords of a spanning tree, but this can be achieved by changing the graph: just double every edge and subdivide the new edges once. The new edges then form a normal spanning tree in the resulting graph~$G$, whose chords are the original edges of our~$T_2$, and $\sigma$ is still a (null-homotopic) loop in~$|G|$.}

We shall therefore define the reduction of infinite words differently, in a non-recursive way: just as a homotopy can shrink a loop simultaneously (rather than recursively) in many places at once, our reduction `steps' will be ordered linearly but not be well-ordered. This definition is less straightforward, but it has an important property: as for finite~$G$, our notion of reduction will be purely combinatorial and make no reference to the topology of~$|G|$.

The main step then will be to show that $\pi_1(|G|)$ embeds as a subgroup in the group of reduced words, and how. We shall see that, as in the finite case, the map assigning to a path in $|G|$ its trace of chords and reducing that trace is well defined on homotopy classes. In~\cite{FundGp} we shall prove that the map it induces on these classes is injective. Then $\pi_1(|G|)$~can be viewed as a subgroup of the group of reduced infinite words, which in turn can be viewed as a subgroup of the inverse limit of the free groups with generators the finite sets of oriented chords of any fixed normal spanning tree (Theorem~\ref{pi1thm}).

Let us make all this precise. Let $G$ be a locally finite connected graph, fixed throughout this section. Let $T$ be a fixed normal spanning tree of~$G$, and write $\Tbar$ for its closure $T\cup\Omega(G)$ in~$|G|$. Unless otherwise mentioned, the endpoints of all paths considered from now on will be vertices or ends, and any homotopies between paths will be relative to~$\{0,1\}$.

When we speak of `the passes' of a given path~$\sigma$, without referring to any particular edges, we shall mean the passes of $\sigma$ through chords of~$T$. If $T$ has only finitely many chords, then $|G|$ is homotopy equivalent to a finite graph, by Lemma~\ref{clNST}. Let us therefore assume that $T$ has infinitely many chords. Enumerate these as $e_0,e_1,\dots$, and denote their natural orientations as $\ve_0,\ve_1,\ldots$. The circles in $T+e_i$ are the \emph{fundamental circles} of~$e_i$. (Up to orientation and reparametrization, there is a only one such circle for every chord.)

Let us call the elements of the set
 $$A:= \{\ve_0, \ve_1, \dots\}\cup\{\ev_0, \ev_1,\dots\}$$
\emph{letters}, and two letters $\ve_i,\ev_i$ \emph{inverse} to each other. A~\emph{word} in~$A$ is a map $w\colon S\to A$ from a totally ordered countable set~$S$, the set of \emph{positions} of (the letters used by)~$w$, such that $w^{-1}(a)$ is finite for every $a\in A$. The only property of $S$ relevant to us is its order type, so two words $w\colon S\to A$ and $w'\colon S'\to A$ will be considered the same if there is an order-preserving bijection $\varphi\colon S\to S'$ such that $w=w'\circ \varphi$. If $S$ is finite, then $w$ is a \emph{finite} word; otherwise it is~\emph{infinite}. The \emph{concatenation} $w_1 w_2$ of two words is defined in the obvious way: we assume that their sets $S_1,S_2$ of positions are disjoint, put $S_1$ before~$S_2$ in~$S_1\cup S_2$, and let $w_1w_2$ be the combined map $w_1\cup w_2$. For $I\sub \N$ we let
 $$A_I\assign\{\ve_i \mid i\in I\}\cup \{\ev_i \mid i\in I\}\,,$$
and write $w\restr I$ as shorthand for the restriction $w\restr w^{-1}(A_I)$. Note that if $I$ is finite then so is the word~$w\restr I$, since $w^{-1}(a)$ is finite for every~$a$.

An \emph{interval} of~$S$ is a~sub\-set $S'\sub S$ closed under betweenness, i.e., such that whenever $s'<s<s''$ with $s',s''\in S'$ then also $s\in S'$. The most frequently used intervals are those of the form $[s',s'']_S\assign\{s\in S \mid {s' \le s \le s''}\}$  and $(s',s'')_S\assign\{s\in S \mid s'<s<s''\}$. If $(s',s'')_S=\es$, we call $s',s''$ \emph{adjacent} in~$S$.

A~\emph{reduction} of a finite or infinite word $w\colon S\to A$ is a totally ordered set $R$ of disjoint 2-element subsets of~$S$ such that the two elements of each $p\in R$ are adjacent in $S\sm\bigcup\{q\in R \mid q<p\}$ and are mapped by $w$ to inverse letters~$\ve_i,\ev_i$. We say that \emph{$w$ reduces to} the word $w \restr (S\sm\bigcup R)$. If $w$ has no nonempty reduction, we call it \emph{reduced}.

Informally, we think of the ordering on~$R$ as expressing time. A~reduction of a finite word thus recursively deletes cancelling pairs of (positions of) inverse letters; this agrees with the usual definition of reduction in free groups. When $w$ is infinite, cancellation no longer happens `recursively in time', because $R$ need not be well ordered.

As is well known, every finite word $w$ reduces to a unique reduced word, which we denote as~$r(w)$. Note that $r(w)$ is unique only as an abstract word, not as a restriction of~$w$: if $w = \ve_1\ev_1\ve_1$ then $r(w) = \ve_1$, but this letter $\ve_1$ may have either the first or the third position in~$w$. The set of reduced finite words forms a group, with multiplication defined as $(w_1,w_2)\mapsto r(w_1 w_2)$, and identity the empty word~$\es$. This is the free group with free generators $\ve_0,\ve_1,\dots$ and inverses $\ev_0,\ev_1\dots$. For finite $I\sub \N$, the subgroup
 $$F_I := \{w\mid \im w\sub A_I\}$$
 is the free group on $\{\ve_i \mid i\in I\}$.

Consider a word~$w$, finite or infinite, and $I\sub \N$. It is easy to check the following:
\begin{equation}\begin{minipage}[c]{0.72\textwidth}\label{inducedreduction}\em
If $R$ is a reduction of~$w$ then $\big\{\{s,s'\}\in R\mid w(s)\in A_I\big \}$, with the ordering induced from~$R$, is a reduction of~$w\restr I$.
  \end{minipage}\ignorespacesafterend\end{equation}

\noindent
   In particular:
\begin{equation}\begin{minipage}[c]{0.72\textwidth}\label{inducedreductioninformal}\em
Any result of first reducing and then restricting a word can also be obtained by first restricting and then reducing it.\looseness=-1
  \end{minipage}\ignorespacesafterend\end{equation}

Our aim is to define a reduction map $r$ also for infinite words, so that $(w_1,w_2)\mapsto r(w_1 w_2)$ makes the set of reduced (finite or infinite) words into a group~$F_\infty$.\footnote{The notation $F_{\infty}$ is chosen in analogy to the groups $F_I$ of finite words, but note that $F_{\infty}$ will \emph{not} be a free group.} This group $F_\infty$ will contain $\pi_1(|G|)$ as a subgroup by an embedding $\langle\alpha\rangle\mapsto r(w_\alpha)$, where $w_\alpha$ is the word of chords traced out by~$\alpha$.

To define such a reduction map~$r$, we need to show that every infinite word reduces to a unique reduced word. Existence is immediate:

\begin{lemma}\label{lemma:reduce}
  Every word reduces to some reduced word.
\end{lemma}

\begin{proof}
  Let $w\colon S\to A$ be any word. By Zorn's Lemma there is a maximal reduction $R$ of~$w$. Since $R$ is maximal, the word $w\restr (S\sm\bigcup R)$ is reduced.
\end{proof}

To prove uniqueness, we begin with a characterization of the reduced words.
Let $w\colon S\to A$ be any word. If $w$ is finite, call a position $s\in S$ \emph{permanent} in~$w$ if it is not deleted in any reduction, i.e., if $s\in S\sm\bigcup R$ for every reduction $R$ of~$w$. If $w$ is infinite, call a position $s\in S$ \emph{permanent} in $w$ if there exists a finite $I\sub\N$ such that $w(s)\in A_I$ and $s$ is permanent in $w\restr I$. By~\eqref{inducedreductioninformal}, a permanent position of $w\restr I$ is also permanent in $w\restr J$ for all finite $J\supseteq I$. The converse, however, need not hold: it may happen that $\{s,s'\}$ is a pair (`of cancelling positions') in a reduction of~$w\restr I$ but $w\restr J$ has a letter from $A_J\sm A_I$ whose position lies between $s$ and~$s'$, so that $s$ and $s'$ are permanent in~$w\restr J$.

\begin{lemma}\label{permanent}
  A word is reduced if and only if all its positions are permanent.
\end{lemma}

\begin{proof}
  The assertion is clear for finite words, so let $w\colon S\to A$ be an infinite word. Suppose first that all positions of $w$ are permanent. Let $R$ be any reduction of~$w$; we will show that $R=\es$.
Let $s$ be any position of~$w$. As $s$ is permanent, there is a finite $I\sub\N$ such that $w(s)\in A_I$ and $s$ is not deleted in any reduction of~$w\restr I$. By~\eqref{inducedreduction}, the pairs in $R$ whose elements map to~$A_I$ form a reduction of~$w\restr I$, so $s$ does not lie in such a pair. As $s$ was arbitrary, this proves that $R=\es$.

Now suppose that $w$ has a non-permanent position~$s$. We shall construct a non-trivial reduction of~$w$. For all $n\in\N$ put $S_n := \{s\in S \mid w(s)\in A_{\{0,\dots,n\}}\}$; recall that these are finite sets. Write $w_n$ for the finite word $w\restr I$ with $I=\{0,\dots,n\}$, the restriction of $w$ to~$S_n$. For any reduction $R$ of~$w_{n+1}$, the set $R^- := \big\{\{t,t'\}\in R\mid t,t'\in S_n\big\}$ with the induced ordering is a reduction of~$w_n$, by~\eqref{inducedreduction}.

Pick $N\in\N$ large enough that $s\in S_N$. Since $s$ is not permanent in~$w$, every $w_n$ with $n\ge N$ has a reduction in which $s$ is deleted. As $w_n$ has only finitely many reductions, K\"onig's infinity lemma~\cite[Lem. 8.1.2]{DiestelBook05} gives us an infinite sequence $R_N, R_{N+1}, \dots$ in which each $R_n$ is a reduction of~$w_n$ deleting~$s$, and $R_n = R_{n+1}^-$ for every~$n$. Inductively, this implies:
  \begin{txteq}\label{pairs}
  For all $m\le n$, we have $R_m = \big\{\{t,t'\}\in R_n\mid t,t'\in S_m\big\}$, and the ordering of $R_m$ on this set agrees with that induced by~$R_n$.
   \end{txteq}
Let $s'\in S$ be such that $\{s,s'\}\in R_n$ for some~$n$; then $\{s,s'\}\in R_n$ for every $n\ge N$, by~\eqref{pairs}.

Our sequence $(R_n)$ divides the positions of~$w$ into two types. Call a position $t$ of $w$ \emph{essential} if there exists an $n\ge N$ such that $t\in S_n$ and $t$ remains undeleted in~$R_n$; otherwise call $t$ \emph{inessential}. Consider the set
 $$R\assign\bigcup_{m\ge N}\bigcap_{n\ge m}R_n$$
 of all pairs of positions of $w$ that are eventually in~$R_n$. Let $R$ be endowed with the ordering $p < q$ induced by all the orderings of $R_n$ with $n$ large enough that $p,q\in R_n$; these orderings are compatible by~\eqref{pairs}. Note that $R$ is non-empty, since it contains~$\{s,s'\}$. We shall prove that $R$ is a reduction of~$w$.

We have to show that the elements of each $p\in R$, say $p = \{t_1,t_2\}$ with $t_1 < t_2$, are adjacent in $S\sm\bigcup\{q\in R \mid q<p\}$. Suppose not, and pick $t\in (t_1,t_2)_S\sm\bigcup\{q\in R \mid q<p\}$. If $t$ is essential, then $t$ is a position of $w_n$ remaining undeleted in~$R_n$ for all large enough~$n$. But then $\{t_1,t_2\}\notin R_n$ for all these~$n$, contradicting the fact that $\{t_1,t_2\}\in R$. Hence $t$ is inessential. Then $t$ is deleted in every $R_n$ with $n$ large enough. By~\eqref{pairs}, the pair $\{t,t'\}\in R_n$ deleting~$t$ is the same for all these~$n$, so $\{t,t'\} =: p'\in R$. By the choice of~$t$, this implies $p'\not< p$. For $n$ large enough that $p,p'\in R_n$, this contradicts the fact that $t_1,t_2$ are adjacent in $S_n\sm\bigcup\{q\in R_n, q<p\}$, which they are since $R_n$ is a reduction of~$w_n$.
\end{proof}

Note that a word can consist entirely of non-permanent positions and still reduce to a non-empty word: the word $\ve_1\ev_1\ve_1$ is again an example.

Lemma~\ref{permanent} offers an easy way to check whether an infinite word is reduced. In general, it can be hard to prove that a given word $w$ has no non-trivial reduction, since this need not have a `first' cancellation. By Lemma~\ref{permanent} it suffices to check whether every position becomes permanent in some large enough but finite~$w\restr I$.

Similarly, it can be hard to prove that two words reduce to the same word. The following lemma provides an easier way to do this, in terms of only the finite restrictions of the two words:

\begin{lemma}\label{reduceonI}
  Two words $w,w'$ can be reduced to the same (abstract) word if and only if $r(w\restr I) = r(w'\restr I)$ for every finite $I\sub\N$.
\end{lemma}

\begin{proof}
The forward implication follows easily from~\eqref{inducedreductioninformal}. Conversely, suppose that $r(w\restr I) = r(w'\restr I)$ for every finite $I\sub\N$. By Lemma~\ref{lemma:reduce}, $w$~and $w'$ can be reduced to reduced words $v$ and~$v'$, respectively. Our aim is to show that $v=v'$, that is to say, to find an order-preserving bijection $\phi\colon S\to S'$ between the domains $S$ of $v$ and $S'$ of~$v'$ such that $v = v'\circ\phi$. For every finite~$I$, our assumption and the forward implication of the lemma yield
 $$r(v\restr I)  = r(w\restr I) = r(w'\restr I) = r(v'\restr I)\,.$$
 Hence for every possible domain $S_I\sub S$ of $r(v\restr I)$ and every possible domain $S'_I\sub S'$ of~$r(v'\restr I)$ there exists an order isomorphism $S_I\to S'_I$ that commutes with $v$ and~$v'$. For every~$I$, there are only finitely many such maps $S_I\to S'_I$, since there are only finitely many such sets $S_I$ and~$S'_I$. And for $I\sub J$, every such map $S_J\to S'_J$ induces such a map $S_I\to S'_I$ with $S_I\sub S_J$ and $S'_I\sub S'_J$, by \eqref{inducedreductioninformal}. Hence by the infinity lemma there exists a sequence $\phi_0\sub \phi_1\sub\dots$ of such maps $\phi_n\colon S_{\{0,\dots,n\}}\to S'_{\{0,\dots,n\}}$, whose union~$\phi$ maps all of $S$ onto~$S'$, since by Lemma~\ref{permanent} every position of $v$ and of $v'$ is permanent.
\end{proof}

With Lemma~\ref{reduceonI} we are now able to prove:

\begin{lemma}\label{uniquereduced}
  Every word reduces to a unique reduced word. 
\end{lemma}

\begin{proof}
  By Lemma~\ref{lemma:reduce}, every word $w$ reduces to some reduced word~$w'$. Suppose there is another reduced word $w''$ to which $w$ can be reduced. By the easy direction of Lemma~\ref{reduceonI}, we have
 $$r(w'\restr I) = r(w\restr I) = r(w''\restr I)$$
 for every finite~$I\sub\N$. By the non-trivial direction of Lemma~\ref{reduceonI}, this implies that $w'$ and~$w''$ can be reduced to the same word. Since $w'$ reduces only to~$w'$ and $w''$ reduces only to~$w''$, this must be the word $w'=w''$.
\end{proof}

As in the case of finite words, we denote the unique reduced word that a word $w$ reduces to by~$r(w)$. The set of reduced words now forms a group $F_\infty$, with multiplication defined as $(w_1,w_2)\mapsto r(w_1w_2)$, identity the empty word~$\es$, and the inverse $w^-$ of $w\colon S\to A$ defined as the map on the same~$S$, but with the inverse ordering, that satisfies $\{w(s),w^-(s)\} = \{\ve_i,\ev_i\}$ for some~$i$ for every~${s\in S}$. (Thus, $w^-$~is $w$ taken backwards and with reverse chord orientations.) Note that the proof of associativity requires an application of Lemma~\ref{uniquereduced}.

In~\cite{FundGp}, we prove that $F_\infty$ embeds canonically in the inverse limit of the groups~$F_I$ (which by \eqref{inducedreductioninformal} form an inverse system with respect to restriction), and that $\pi_1(|G|)$ embeds canonically in~$F_\infty$. All we need here is the existence of that particular homomorphism $\pi_1(|G|)\to F_\infty$, which is defined as follows.

Every path $\sigma$ in $|G|$ defines a word~$w_{\sigma}$ by its passes through the chords of~$T$. Formally, we take as $S$ the set of the domains $[a,b]$ of passes of~$\sigma$, ordered naturally as disjoint subsets of~$[0,1]$, and let $w_{\sigma}$ map every $[a,b]\in S$ to the directed chord that $\sigma$ traverses on~$[a,b]$. We call $w_{\sigma}$ the \emph{trace} of $\sigma$.

\begin{lemma}\label{lemma:reducible1}
  The traces of homotopic paths in $|G|$ reduce to the same word.
   \end{lemma}

\begin{proof}
We first consider homotopic loops $\alpha\sim\beta$ in~$|G|$, based at the same vertex or end. We wish to show that $r(w_\alpha) = r(w_\beta)$. By Lemma~\ref{reduceonI} it suffices to show that $r(w_{\alpha}\restr I) = r(w_{\beta}\restr I)$ for every finite $I\subset\N$. Consider the space obtained from $|G|$ by attaching a 2-cell to $|G|$ for every $j\notin I$, by an injective attachment map from the boundary of the 2-cell onto the unique image of the fundamental circles of~$e_j$. This space deformation retracts onto $\Tbar\cup\{e_i\mid i\in I\}$, and hence by Lemma~\ref{clNST} onto a finite graph $G_I$ consisting of a subtree of~$T$ and the chords $e_i$ with $i\in I$. Composing $\alpha$ and $\beta$ with the map $|G|\to G_I$ in which this retraction ends yields homotopic loops $\alpha'$ and $\beta'$ in~$G_I$, whose traces in $F_I$ are $w_{\alpha'} = w_{\alpha}\restr I$ and $w_{\beta'} = w_{\beta}\restr I$. Since $\langle\gamma\rangle\mapsto r(w_\gamma)$ is known to be well defined for finite graphs, we deduce that
 $$r(w_{\alpha}\restr I) = r(w_{\alpha'}) = r(w_{\beta'}) = r(w_{\beta}\restr I)\,.$$
 This completes the proof of the lemma for loops. The general case follows, since paths joining the same two vertices or ends can be made into loops by appending a path in $\Tbar$ joining their endpoints, which does not change their traces.
\end{proof}

Lemma~\ref{lemma:reducible1} implies in particular that the map $\langle\alpha\rangle\mapsto r(w_\alpha)$ from $\pi_1(|G|)$ to $F_\infty$ is well defined. By~\eqref{inducedreductioninformal}, it is a homomorphism. In~\cite{FundGp} we  prove that it is injective---the converse of Lemma~\ref{lemma:reducible1}---and determine its image. All in all, the fundamental group of $|G|$ can be described combinatorially by canonical group embeddings $\pi_1(|G|)\to F_\infty\to \varprojlim F_I$ as follows (see~\cite{FundGp} for precise definitions):

\begin{theorem}[\cite{FundGp}]\label{pi1thm}
   Let $G$ be a locally finite connected graph. Let $T$ be a normal spanning tree of~$G$, and let $e_0, e_1,\dots$ be its chords.
\begin{enumerate}[\rm (i)]
  \item The map $\langle\alpha\rangle\mapsto r(w_\alpha)$ is an injective homomorphism from $\pi_1(|G|)$ to the group $F_\infty$ of reduced finite or infinite words in $\{\ve_0, \ve_1, \dots\}\cup\{\ev_0, \ev_1,\dots\}$, with image the set of words whose monotonic subwords converge in~$|G|$.
   \item The homomorphisms $w\mapsto r(w\restr I)$ from $F_\infty$ to~$F_I$ embed $F_\infty$ as a subgroup in~$\varprojlim F_I$. It consists of those elements of~$\varprojlim F_I$ whose projections $r(w\restr I)$ use each letter only boundedly often. (The bound may depend on the letter.)
\end{enumerate}
\end{theorem}

Using this characterization of $\pi(|G|)$, it is not hard to see that $\vCC(G)$ is the \emph{strong abelianization} of $\pi(|G|)$: the quotient of $\pi(|G|)$ obtained by identifying classes $\langle\alpha\rangle,\langle\beta\rangle$ whenever $r(w_{\alpha})$ and $r(w_{\beta})$ use each letter the same number of times (see~\cite[Sec.~4]{CannonConnerER} for a formal definition of strong abelianization).

\section{Distinguishing boundaries from other chains}\label{windup}

In this section we wind up our proof of Theorem~\ref{compthm}, picking up the thread from Section~\ref{comp}. There we had defined a group homomorphism ${f\colon H_1(|G|)\to\vEE(G)}$, shown that its image is~$\vCC(G)$ (Lemmas \ref{rangeC} and~\ref{surjectivity}), and seen that $f$ is injective if $G$ contains only finitely many circuits. The assertion left to prove is that $f$ is not injective if $G$ contains infinitely many circuits, which we now assume. 

Let $T$ be a normal spanning tree of~$G$. Each of the infinitely many circuits in $G$ is a finite sum (mod~2) of distinct fundamental circuits of~$T$~\cite[Thm. 1.9.6]{DiestelBook05}. Therefore $T$~has infinitely many chords, $e_0,e_1,\dots$ say. Since $|G|$ is compact, there is a sequence $\ve_{i_0},\ve_{i_1},\dotsc$ of chords whose first points converge to an end $\omega$ of~$G$. There exists a loop $\rho$ in~$|G|$, based at a vertex, that traverses $\ve_{i_0},\ve_{i_1},\dotsc,\ev_{i_0},\ev_{i_1},\dotsc$ in this order and runs otherwise along~$\Tbar$. (Thus, $\rho~$~starts with passes through $\ve_{i_0},\ve_{i_1},\dots$, interspersed with finite segments of~$T$ between the endpoints of these passes, until it reaches~$\omega$, from where it returns along $\Tbar$ to the starting vertex of~$\ev_{i_0}$; it then traverses $\ev_{i_0},\ev_{i_1},\dots$ interspersed with connecting segments of $T$ to reach $\omega$ a second time, and finally returns from there along~$\Tbar$ to its starting vertex. Note that the convergence of $e_{i_0}, e_{i_1},\dots$ is essential for $\rho$ to be a path: there is no path in $|G|$ through an $\omega$-sequence of chords that does not converge.) Since $\rho$ traverses the chords of $T$ equally often in both directions, Theorem~\ref{orthogonal} and Lemmas \ref{traverse} and~\ref{rangeC} imply that $\rho$ also traverses the edges of $T$ equally often in both directions.

Hence $f([\rho])=0\in\vCC(G)$. To complete the proof that $f$ is not injective, and thereby the proof of Theorem~\ref{compthm}, we show that $[\rho]\ne0$, i.e.\ that $\rho$ is not the boundary of any 2-chain. In order to do so, we shall use our results from Section~\ref{pi1sec} to define an invariant of 1-chains that can distinguish $\rho$ from boundaries. As in Section~\ref{pi1sec} we consider only paths whose boundary points are vertices or ends, so our invariant will be defined only for chains of 1-simplices with this property. However, it is easy to see that this entails no loss of generality.%
\footnote{Indeed, if $\rho = \sum\lambda_n\partial\tau_n$ for 2-simplices~$\tau_n$, we can modify each $\tau_n$ into another 2-simplex $\tau'_n$ whose 0-faces are vertices or ends, ans such that $\rho = \sum\lambda_n\partial\tau_n$. For every inner point $x$ of an edge $e_x=u_x v_x$ in $|G|$ pick a fixed path $\pi_x$ from $x$ to~$v_x$ (say). Then append to every 1-simplex $\sigma$ occurring in the boundary of a $\tau_n$ and ending in such a point $x$ the path~$\pi_x$ after~$x$, turning $\sigma$ into a path $\sigma'$ between two vertices by appending at most two such paths~$\pi_x$. Then if $\partial\tau_n = \sigma_0 - \sigma_1 + \sigma_2$, say, it is easy to see that also $\sigma'_0 - \sigma'_1 + \sigma'_2$ is the boundary of a 2-simplex~$\tau'_n$. And clearly $\rho = \sum\lambda_n\partial\tau_n$ implies that also $\rho = \sum\lambda_n\partial\tau'_n$, since we modified only 1-simplices that cancelled out anyway in this sum.}

We need some more notation. Given $k\in\N$ and a reduced word $w\colon S\to A$ (where $A=\{\ve_0,\ev_0,\ve_1,\ev_1,\dotsc\}$ as before), write $n^+(w,k)$ for the number of intervals in $S$ which can be written as $\{s_0,s_1,\dotsc\}$ with $s_0<s_1<\dots$ and $w(s_j)=\ve_{i_{k+j}}$ for every $j\in\N$. (This number exists: there are at most $|w^{-1}(\ve_{i_k})|$ such intervals, and this number is finite by our definition of `word'.) Put
 $$n(w,k)\assign n^+(w,k)-n^+(w^-,k)\in\Z\,.$$
 (Recall that $w^-$ is $w$ backwards with inverse letters, so $n^+(w^-,k)$ counts the intervals in $S$ which can be written as $\{s_0,s_1,\dotsc\}$ with $s_0>s_1>\dots$ and $w_{\sigma}(s_j)=\ev_{i_{k+j}}$ for every $j\in\N$.)

Given $k\in\N$ and a path $\sigma$ in~$|G|$, let $N(\sigma,k)\assign n(r(w_{\sigma}),k)$. Given a 1-chain $\phi = \sum_n \lambda_n\sigma_n$, let $N(\phi,k) := \sum_n \lambda_n N(\sigma_n,k)$ for every fixed~$k$, and put 
 $$\textstyle N(\phi) := \min_k |N(\phi,k)|\,.$$
 Unlike $N(\phi,k)$ for fixed~$k$, the function $N$ is not a homomorphism. Nevertheless, it will help us distinguish our special path $\rho$ from boundaries: we shall prove that $N$ vanishes on boundaries, while clearly $N(\rho) = 1$. (Indeed, the word $w_{\rho}$ is easily seen to be reduced (cf.\ Lemma~\ref{permanent}); hence $N(\rho,k)=n(w_{\rho},k)=1$ for all~$k$, since $n^+(w_{\rho},k) = 1$ while $n^+(w_{\rho}^-,k) = 0$.)

We begin by noting a property of the function $n(w,k)$:
\begin{txteq}\label{subdivideN}
  If $w=w_1w_2$ is a reduced word, there exists a $k\in\N$ such that $n(w,\ell)=n(w_1,\ell)+n(w_2,\ell)$ for all $\ell\ge k$.
\end{txteq}
Indeed, denote the domains of $w_1$ and $w_2$ by $S_1$ and~$S_2$ (chosen disjoint); then the domain of $w$ is $S=S_1\cup S_2$, with $S_1$ preceding~$S_2$. If $S_1$ has a largest element, $s_1$~say, choose $k$ large enough that $w(s_1)\notin\{\ve_{i_k},\ev_{i_k},\ve_{i_{k+1}},\ev_{i_{k+1}},\dotsc\}$. Then for every $\ell\ge k$ none of the intervals in $S$ counted by $n(w,\ell)$ meets both $S_1$ and $S_2$, since these intervals cannot contain~$s_1$. Hence every such interval is either an interval of $S_1$ or one of~$S_2$, so $n(w,\ell)=n(w_1,\ell)+n(w_2,\ell)$ as desired. On the other hand if $S_1$ has no largest element, then no interval in $S$ that meets both $S_1$ and $S_2$ can be written as $\{s_0,s_1,\dotsc\}$ with $s_0<s_1<\dots$ or $s_0>s_1>\dots$, so none of the intervals counted by $n(w,k)$ for any $k$ meets both $S_1$ and~$S_2$. Hence, in this case, $n(w,k)=n(w_1,k)+n(w_2,k)$ for all $k$.

For our proof that $N$ vanishes on boundaries~$\phi$, it suffices to show that every 2-simplex $\tau$ satisfies $N(\partial\tau,k) = 0$ for large enough~$k$: then $N(\phi,k)=0$ for some (large)~$k$, and hence $N(\phi)=0$ as claimed. So consider a 2-simplex~$\tau$, with boundary $\partial\tau = \sigma_2 - \sigma_1 + \sigma_0$ denoted so that $\sigma_2$ ends at the starting vertex of~$\sigma_0$. Write $w_i := r(w_{\sigma_i})$ for the words to which the traces of the $\sigma_i$ reduce ($i=1,2,3$), and $w_{20} := r(w_{\sigma_{20}})$, where $\sigma_{20}:= \sigma_2 \sigma_0$ is the path consisting of $\sigma_2$ followed by~$\sigma_0$.

Note that $w_{20} = r(w_2 w_0)$. Indeed, we can reduce $w_{\sigma_{20}}$ by first applying to $w_{\sigma_2}\sub w_{\sigma_{20}}$ the reduction that turns $w_{\sigma_2}$ into~$w_2$, and then apply to $w_{\sigma_0}\sub w_{\sigma_{20}}$ the reduction that turns $w_{\sigma_0}$ into~$w_0$. Together this is a reduction of $w_{\sigma_{20}}$ to $w_2 w_0$. Let $R$ be a reduction of $w_2 w_0$ to~$r(w_2 w_0)$. Since we started with~$w_{\sigma_{20}}$, the reduced word $r(w_2 w_0)$ we end up with is $r(w_{\sigma_{20}}) = w_{20}$ (Lemma~\ref{uniquereduced}).

Let us look at what $R$ does. Since $w_2$ and $w_0$ are both reduced, every pair of positions in $R$ has one position in $w_2$ and the other in~$w_0$. Hence if $w$ denotes the subword of $w_2$ whose positions are deleted by~$R$, we have found reduced words $w, w'_2, w'_0$ such that
 $$w_2 = w'_2 w\quad\text{and}\quad w_0 = w^-w'_0\quad\text{and}\quad  w_{20} = w'_2 w'_0\,.$$
 By~\eqref{subdivideN}, therefore, we have for all large enough $k$
\begin{align*}
  n(w_2,k) &= n(w'_2,k) + n(w,k)\\
  n(w_0,k) &= n(w^-,k) + n(w'_0,k)\\
  n(w_{20},k) &= n(w'_2,k) + n(w'_0,k)\,.
  \end{align*}
As $n(w^-,k) = -n(w,k)$, we deduce that
 $$n(w_2,k) + n(w_0,k) - n(w_{20},k) = 0$$
for all these~$k$.

Since $\sigma_{20}$ is homotopic to~$\sigma_1$ (across~$\tau$), Lemma~\ref{lemma:reducible1} implies that $w_{20} = w_1$. We therefore deduce that
 $$N(\partial\tau,k) = N(\sigma_2,k) + N(\sigma_0,k) - N(\sigma_1,k) = 0$$
for all large enough~$k$, as desired. This completes the proof of Theorem~\ref{compthm}.

\section{A new homology for non-compact complexes}\label{new}

In this section we introduce a modification of singular homology for any topological space $X$ embedded in a larger space~$\hat X$, which assigns a special role to the points in $\hat X\sm X$. The kind of example we have in mind is that $X$ is locally compact and $\hat X$ a compactification of~$X$; see e.g.~\cite{AbelsStrantzalos} for more on such spaces. For this reason we shall call the points in $\hat X\sm X$ the \emph{ends} of~$X$; but formally we make no assumptions other than that $X\sub \hat X$. For compact $X=\hat X$ our homology will coincide with standard singular homology. When $X$ is a graph and $\hat X$ its Freudenthal compactification, the first group of our homology will be canonically isomorphic to the group of the topological cycle space of~$X$.

Although our chains, cycles etc.\ will live in~$\hat X$, we shall denote their groups as $C_n(X)$, $Z_n(X)$ etc, with reference to $X$ rather than~$\hat X$: this is because ends will play a special role, so the information of which points of $\hat X$ are ends must be encoded in the notation for those groups.

Let us call a family $(\sigma_i\mid i\in I)$ of singular $n$-simplices in $\hat X$ \emph{good} if
\begin{enumerate}[(i)]
\item $(\sigma_i\mid i\in I)$ is locally finite in~$X$, that is, every $x\in X$ has a neighbourhood in $X$ that meets the image of $\sigma_i$ for only finitely many~$i$;
\item every $\sigma_i$ maps the 0-faces of~$\Delta^n$ (the standard $n$-simplex) to~$X$.
\end{enumerate}
Note that if $X$ is locally compact, then (i)~is equivalent to asking that every compact subspace of $X$ meets the image of $\sigma_i$ for only finitely many~$i$.
Condition~(ii), like~(i), underscores that ends are not treated on a par with the points in~$X$: we allow them to occur on infinitely many~$\sigma_i$ (which (i) forbids for points of~$X$), but not in the fundamental role of images of 0-faces: all simplices must be `rooted' in~$X$. Since $X$ is, by assumption, a countable union of compact spaces, (i)~and (ii) together imply that good families are countable, i.e.\ that $|I|\le\aleph_0$.

When $(\sigma_i\mid i\in I)$ is a good family, any formal linear combination $\sum_{i\in I} \lambda_i \sigma_i$ with all $\lambda_i\in\Z$ is an \emph{$n$-chain in~$X$}. Since only finitely many $\sigma_i$ in a good family can be equal, we can add up the coefficients of equal terms and thus assume if desired that the $\sigma_i$ in an $n$-chain are pairwise distinct; however, we shall \emph{not} normally assume this. We write $C_n(X)$ for the group of all $n$-chains in~$X$, and $C'_n(X)$ for its subgroup of finite $n$-chains. The boundary operator $\partial_n\colon C_n\to C_{n-1}$ is defined as usual by extending linearly from~$\partial\sigma_i$. Note that $\partial$ is well defined (i.e., that it preserves the local finiteness required of chains), and $\partial^2 = 0$. Chains in $\im\partial$ will be called \emph{boundaries}.

As $n$-cycles, we do \emph{not} take the entire kernel of~$\partial_n$. Rather, we define $Z'_n(X) := \ke (\partial_n\restr C'_n(X))$, and let $Z_n(X)$ be the set of those $n$-chains that are sums of such finite cycles:
 $$Z_n (X) := \Big\{\phi\in C_n(X)\mid \phi = \sum_{j\in J} z_j
    \emtext{ with } z_j\in Z'_n(X)\ \forall j\in J\Big\}\,.$$
More precisely, an $n$-chain $\phi\in C_n(X)$ shall lie in $Z_n(X)$ if we can write it as $\phi = \sum_{i\in I}\lambda_i\sigma_i$ in such a way that $I$ admits a partition into finite sets~$I_j$ ($j\in J$) with $z_j := \sum_{i\in I_j} \lambda_i\sigma_i \in Z'_n(X)$ for every $j\in J$. Any such representation of $\phi$ as a formal sum will be called a \emph{standard representation} of~$\phi$ \emph{as a cycle}.%
\footnote{Since the $\sigma_i$ need not be distinct, $\phi$~has many representations as a formal sum. Not all of these need admit a partition as indicated---see below.}
We call the elements of $Z_n(X)$ the \emph{$n$-cycles} of~$X$.

The chains in $B_n(X) := \im\partial_{n+1}$ then form a subgroup of~$Z_n(X)$: by definition, they can be written as $n$-chains $\sum_{j\in J} z_j$ where each $z_j$ is the (finite) boundary of a singular $(n+1)$-simplex. We therefore have homology groups
 $$H_n(X) := Z_n(X)/B_n(X)$$
as usual. Alternatively, we may study the subgroups
 $$H'_n (X) := \big\{[z] : z\in Z'_n\big\}$$
of~$H_n(X)$ formed by the homology classes of finite cycles.

Note that if $X$ is compact, then all good families and hence all chains are finite, so the homology defined above coincides with the usual singular homology.
The characteristic feature of this homology is that while infinite cycles are allowed, they are always of `finite character': in any standard representation of an infinite cycle, every finite subchain is contained in a larger finite subchain that is already a cycle.

To illustrate all this let us look at a simple example,%
   \footnote{Another example will be given in~\eqref{CancellingBoundaries} on page~\pageref{CancellingBoundaries}, which the reader is invited to skip to now.}
the \emph{double ladder}. This is the 2-ended graph $G$ with vertices $v_n$ and~$v'_n$ for all integers~$n$, and with edges $e_n$ from $v_n$ to~$v_{n+1}$, edges $e'_n$ from $v'_n$ to~$v'_{n+1}$, and edges $f_n$ from $v_n$ to~$v'_n$. The 1-simplices corresponding to these edges, oriented in their natural directions, are $\theta_{e_n}$, $\theta_{e'_n}$ and~$\theta_{f_n}$. For the infinite chains $\phi:= \sum \theta_{e_n}$, $\phi':= \sum \theta_{e'_n}$ and $\psi := \phi-\phi'$ we have $\partial\phi = \partial\psi = 0$, and neither sum as written above contains a finite cycle. But while it can be shown that $\phi\notin Z_1(G)$, we have $\psi\in Z_1(G)$, because $\psi$ can be rewritten as $\phi = \sum z_n$ with finite cycles $z_n = \theta_{e_n} + \theta_{f_{n+1}} - \theta_{e'_n} - \theta_{f_n}$. By contrast, the unique representation of $\psi$ as a formal sum in which like terms are combined, $\psi = \phi - \phi'$, is not a standard representation of $\psi$ as a cycle. As we shall see in Section~\ref{H1equalsC}, we even have $[\psi]\in H'_1$, although this is not obvious now.

There are some immediate questions that arise from the definitions given above. For example, is $Z_n$ closed under locally finite sums? Is $H'_n = H_n$, i.e., is every homology class represented by a finite cycle? (For $n=1$ we shall see in the next section that it is.) We shall not pursue these questions here, but will explore the properties of our homology (including more fundamental questions than these) in another paper. In the remainder of this paper we show that our homology achieves its aim: for graphs, it captures precisely the oriented topological cycle space.

\section{\boldmath $H_1(G)$ equals $\vCC(G)$}\label{H1equalsC}

In this section we show that, for graphs~$G$, the groups $H_1 (G)$ and $H'_1 (G)$ defined in Section~\ref{new} coincide, and are canonically isomorphic to the topological cycle space $\vCC(G)$ of~$G$.

In analogy to our notation of Section~\ref{comp}, we shall denote this isomorphism by $f\colon H_1(G)\to \vCC(G)$. In our definition of $f$ we shall have to refer to the map which, in Section~\ref{comp}, was denoted as $f\colon H_1 (|G|)\to \vEE(G)$; this map will now be denoted as $f'\colon H_1 (|G|)\to \vEE(G)$. (Recall that $|G|$ denotes the Freudenthal compactification of~$G$, and that $H_1(|G|)$ is its usual first singular homology group.%
	\footnote{We shall use $C_1(G)$, $Z_1(G)$, $B_1(G)$ and $H_1(G)$ to refer to our new homology of~$|G|$ that relies on the information of which points of $|G|$ are ends, while $C_1(|G|)$, $Z_1(|G|)$, $B_1(|G|)$ and $H_1(|G|)$ continue to refer to the usual singular homology of the space~$|G|$.})
When $G$ is finite, our new function $f$ will coincide with~$f'$. 

In order to define~$f$, let $\phi\in Z_1(G)$ be given in any standard representation $\phi = \sum_{i\in I} \lambda_i \sigma_i$ as a cycle, and let $\ve\in\vE$ be any oriented edge. We shall first define $f([\phi])(\ve)\in\Z$ with reference to~$\phi$ and its given representation as a cycle, and then show that our definition does not depend on these choices.

To define~$f([\phi])(\ve)$, we show that for all large enough finite subchains $\phi'\in Z'_1(G)$ of $\phi$ the values of $f'([\phi'])(\ve)$ agree (the homology class $[\phi']$ being taken in~$H_1(|G|)$), and set $f([\phi])(\ve)$ to this common value. Write $I_e$ for the set of those $i\in I$ whose $\sigma_i$ meets~$e$; since $e$ is compact and $(\sigma_i\mid i\in I)$ is a good family, $I_e$~is a finite set.

Let $\pi\colon H_1(S^1)\to\Z$ and $f_e\colon |G|\to S^1$ be defined as in Section~\ref{comp}, and write $\fes\colon C_1(|G|)\to C_1(S^1)$ for the chain map induced by~$f_e$.

\begin{lemma}\label{fextension}
For all finite sets $I'$ such that $I_e\sub I'\sub I$ and $\phi':= \sum_{i\in I'} \lambda_i\sigma_i \in Z_1(|G|)$, the values of $f'([\phi'])(\ve)$ agree.
\end{lemma}

\begin{proof}
Let $\phi_e := \sum_{i\in I_e} \lambda_i\sigma_i$. We show that even if $\phi_e$ is not a cycle in~$|G|$, the chain $\fes(\phi_e)$ is a cycle in~$S^1$ homologous to $\fes(\phi')$ for every $\phi'$ as stated. Then, by definition of~$f'$,
 $$f'([\phi'])(\ve) = \pi( (f_e)_* ([\phi']) )
     = \pi( [\fes(\phi')] ) = \pi( [\fes(\phi_e)] )$$
for all such~$\phi'$, and the result follows.

For a proof of $[\fes(\phi_e)] = [\fes(\phi')]$, note that for all $i\in I\sm I_e$ the map $f_e\circ\sigma_i$ is constant (with value $1\in\CC$). So for such~$i$, $f_e\circ\sigma_i$ is a null-homologous cycle. But $\fes(\phi_e)$ differs from~$\fes(\phi')$, which is a cycle, precisely by the terms $\lambda_i (f_e\circ\sigma_i)$ with $i\in I'\sm I_e$. Hence $\fes(\phi_e)$ too is a cycle,  and it is homologous to~$\fes(\phi')$.
\end{proof}

We now define $f\colon H_1(G)\to\vEE(G)$ by letting $f([\phi])$ map an oriented edge $\ve$ to the common value of $f'([\phi'])(\ve)$ for all $\phi'$ as in Lemma~\ref{fextension}. In order to show that $f$ is well defined, let $\phi\in Z_1(G)$ and $\psi\in B_1(G)$ be given in any standard representations $\phi = \sum_{i\in I} \lambda_i\sigma_i$ and $\psi = \sum_{i\in J} \lambda_i \sigma_i$ with $I\cap J = \es$. We show that $f$ assigns the same value to $[\phi]=[\phi+\psi]$ no matter whether we base its computation on $\phi$ or on~$\phi+\psi$: this proves that $f([\phi])$ depends neither on the choice of $\phi$ as a representative of~$[\phi]$ nor on its representation as $\sum_{i\in I}\lambda_i\sigma_i$.

Given $\ve\in\vE$, let $I_e$ be the set of all $i\in I$ such that $\sigma_i$ meets~$e$, and define $J_e$ likewise. Let $I'\sub I$ and $J'\sub J$ be finite sets containing $I_e$ and~$J_e$, respectively, such that $\phi' := \sum_{i\in I'} \lambda_i\sigma_i \in Z_1 (|G|)$ and $\psi' := \sum_{i\in J'} \lambda_i\sigma_i \in B_1 (|G|)$; such sets exist since $\phi$ and $\psi$ are given in standard representations. Then
 $$f'([\phi'])(\ve) = f'([\phi'+\psi'])(\ve)\,.$$
For our new function~$f$, its value of $[\phi]=[\phi+\psi]$ computed with reference to $\phi$ equals the left-hand side of this equation, while its value computed with reference to $\phi+\psi$ equals the right-hand side. This completes the proof that $f$ is well defined. Note that if $\phi$ is finite, then trivially $f[\phi] = f'[\phi]$, where $[\phi]$ is taken in~$H_1(G)$ and in~$H_1(|G|)$, respectively.

Since $f'$ is a homomorphism with image~$\vCC(G)$ (Section~\ref{comp}), Lemma~\ref{fextension} implies that so is~$f$. Indeed, for a proof that $f([\phi])\in\vCC(G)$ consider the finite oriented cuts $\vF$ of~$G$, and apply Theorem~\ref{orthogonal} to any finite subchain $\phi'$ of $\phi$ containing all the simplices that meet this cut. The proof that $f$ is surjective is the same as in Section~\ref{comp}: every element of $\vCC(G)$ has the form $f([\tau])$ with $\tau$ a single loop. Thus in fact,
\begin{equation*}
\vCC(G)\subseteq f(H'_1(G))\subseteq f(H_1(G))\subseteq\vCC(G)
\end{equation*}
with equality.

The proof of $\im f\subseteq\vCC(G)$ indicated above uses critically that  infinite cycles are of finite character. The following example illustrates that it also depends critically on our rule that chains must be locally finite. Let $G=\RR$, with vertex set~$\Z$. For vertices $m<n$ let $\sigma_{m,n}\colon [0,1]\to G$ interpolate linearly between $m$ and~$n$. Now consider the following sum of boundaries, each of the form $(\sigma_{k,m} + \sigma_{m,n} - \sigma_{k,n})$:
\begin{equation}\label{CancellingBoundaries}
\hskip-.2\textwidth \sum_{i\in\Z}\sigma_{i,i+1}\ = \
   \sum_{i\in\N}(\sigma_{-(i+1),i}+\sigma_{i,i+1}-\sigma_{-(i+1),i+1})
\end{equation}
   \vskip-18pt
$$\hskip.3\textwidth
+\sum_{i\in\N}
   (\sigma_{-(i+2),-(i+1)}+\sigma_{-(i+1),(i+1)}-\sigma_{-(i+2),i+1})\,.$$
The combined sum on the right-hand side of the equation is well defined in that every $\sigma_{m,n}$ occurs at most twice. It also satisfies the requirements for a standard representation of cycles in that it comes partitioned into finite cycles, even boundaries. The left-hand side, which is obtained from the right-hand side by deleting cancelling pairs of simplices, is a well-defined locally finite chain with zero boundary. But it is not (and should not be) an element of~$Z_1(G)$. This is not a contradiction only because its apparent `standard representation' as the combined sum on the right is not locally finite, and hence not a legal chain: the point~0, for example, lies in the image of every~$\sigma_{-n,m}$.

Our final goal is to show that $f$ is injective. For finite~$G$, the standard proof is to rewrite a given cycle $z\in Z_1(G)$ as a homologous sum of simplexes each traversing exactly one edge. If $[z]\in\ke f$, every edge is traversed equally often in both directions, and we can pair up the simplices traversing it accordingly. Each pair is homologous to a boundary, and hence so is~$z$.

The reason why this proof does not work for $f'$ on~$H_1(|G|)$ is that the simplices even in a finite cycle can traverse infinitely many edges. The proof would therefore require us to break up the given finite cycle into a `homologous' infinite chain, which is impossible in~$H_1(|G|)$.

In our new setup, however, this can indeed be done. In fact, it turns out that our restriction that any boundary chains to be added must be locally finite exactly strikes the balance between being restrictive enough to rule out counterexamples like the one above and being general enough to allow the subdivision into chains of single edges even of complicated cycle like our non-injectivity example from Section~\ref{comp}.

This is shown in the following lemma. Although its proof looks somewhat technical, the idea is very simple, so let us describe it informally first. Consider a 1-simplex $\tau$ traversing infinitely many edges. Our task is to `subdivide it infinitely often', into 1-simplices $\sigma_1,\sigma_2,\dots$ each traversing exactly one edge, by adding a locally finite sum of boundaries. We begin by targeting the first pass of $\tau$ through an edge, $e = uv$~say. Let $\sigma_1$ be this pass, and let $\tau'$ and $\tau''$ be the segments of $\tau$ before and after~$\sigma_1$. We now subdivide $\tau$ at $u$ and~$v$: we add to $\tau$ the boundary $\tau' + \sigma_1 + \tau'' - \tau$, to obtain the chain $\tau' + \sigma_1 + \tau''$. Next, we target the second pass of $\tau$ through an edge,~$\sigma_2$. If this is a pass of~$\tau'$, say, with segments $\alpha$ and $\beta$ before and after~$\sigma_2$, we add the boundary $\alpha + \sigma_2 + \beta - \tau'$ to insert $\sigma_2$ into our chain while eliminating~$\tau'$. Doing this for all passes of $\tau$ in turn should leave us at the limit with only the chain $\sigma_1 + \sigma_2 +\dots$, since all other simplices are eliminated again when the earliest pass they contain is targeted. The main task of the formal proof of this, except for the inevitable book-keeping, is to ensure that all the boundaries we add do indeed form a locally finite chain, i.e.\ an element of~$B_1(G)$.

\begin{lemma}\label{InfSubdiv}
For every $z\in Z'_1(G)$ there exist a chain $\phi = \sum_{i\in I} \sigma_i\in Z_1(G)$ and a chain $b\in B_1(G)$ such that $z+b=\phi$, every $\sigma_i$ maps $[0,1]$ homeomorphically to some edge~$e$, and all these edges $e$ as well as the images of the simplices in $b$ are contained in the image of the 1-simplices in~$z$.
\end{lemma}

\begin{proof}
We may clearly assume that $z$ is an elementary cycle consisting of a single loop~$\tau_0$ that is based at a vertex and is not null-homotopic. In particular, $\tau_0$ traverses an edge. Since $\tau_0$ traverses every edge only finitely often (Lemma~\ref{pass}), $\tau_0$~contains only countably many passes through edges, $\pi_1, \pi_2, \ldots$ say, which we reparametrize as maps from~$[0,1]$.

In each of at most $\omega$ steps we shall add to our then current finite cycle
 $$z_n = \sum_{i=1}^n \sigma_i + \sum_{j\in J_n} \tau_j$$
(which initially is $z_0 = \tau_0$) finitely many simplices $\sigma_i$ or $\tau_j$ with coefficients 1 or~$-1$ so that the sum of simplices added lies in~$B_1(G)$. We shall make sure that all these simplices added or deleted form a good family; in particular, their sum will not depend on the order of summation, although this order will help us with our book-keeping. The result will be a chain of the form $\sum_{i\in I} \sigma_i + \sum_{j\in J} \tau_j$ in which every $\tau_j$ is a null-homotopic loop (in particular $0\notin J$) and the $\sigma_i$ are those required in the statement of the lemma.

We shall choose the $z_n$ inductively so as to satisfy the following conditions, which hold for $n=0$ with $J_0 = \{0\}$:

\begin{enumerate}[(i)]
\item\label{gen} $\sigma_1,\ldots,\sigma_n$ and all $\tau_j$ ($j\in J_n$) are paths in $|G|$ between (possibly identical) vertices;
\item\label{inherit} if $n\ge 1$, every $\tau_j$ ($j\in J_n$) is a segment of some $\tau_i$ with $i\in J_{n-1}$;
\item\label{j} if $n\ge 1$, there exists $j(n)$ such that $J_{n-1}\sm J_n = \{j(n)\}$ and the finite chain $b_n := \sigma_n - \tau_{j(n)} + \sum_{j'\in J_n\sm J_{n-1}} \tau_{j'}$ lies in~$B_1(G)$;
\item $\sigma_n$ is homotopic to $\pi_n$ relative to~$\{0,1\}$;
\item\label{passes} suitably reparametrized, $(\pi_{n+1},\pi_{n+2},\dots)$ is the family of all edge-passes of the paths~$\tau_j$ ($j\in J_n$); specifically, the edge-passes in the paths $\tau_{j'}$ with $j'\in J_n\sm J_{n-1}$ are precisely those in $\tau_{j(n)}$ other than~$\pi_n$.
\end{enumerate}

Assuming that $z_{n-1}$ satisfies these conditions, let us define~$z_n$. If $\pi_n$ does not exist, we terminate the construction, putting $I:= \{1,\dots,n-1\}$ and $J:= J_{n-1}$. If it does, then by~\eqref{passes} for $n-1$ there is a unique $j\in J_{n-1}$ such that $\pi_n$ is an edge-pass in~$\tau_j$. The path $\tau_j$ is a concatenation of three segments $\alpha$, $\pi_n$, and~$\beta$, where $\alpha$ and $\beta$ may have trivial domain. Let $J_n$ be obtained from $J_{n-1}$ by removing $j =: j(n)$ and adding new indices $j',j''$ for $\alpha =: \tau_{j'}$ and $\beta =: \tau_{j''}$ whenever these maps are paths (i.e., have non-trivial domain), reparametrizing each to domain~$[0,1]$. Let $\sigma_n$ be an injective path that is homotopic to $\pi_n$ relative to~$\{0,1\}$. Clearly, $z_n$~again satisfies the conditions. If the process continues for $\omega$ steps, we complete it by putting $I:=\N$, and letting $J:= \bigcap_{n\in\N}\bigcup_{k>n} J_k$ consist of those $j$ that are eventually in~$J_n$.

Let us take a look at the simplices $\tau_j$ with $j\in J$. By definition of~$J$, we have $j\in J_n$ for all large enough~$n$. By \eqref{gen} and~\eqref{inherit}, $\tau_j$~is a segment of~$\tau_0$ between two vertices, and by~\eqref{passes} it contains none of the passes $\pi_1,\pi_2,\ldots$. So it does not traverse any edge. Hence,
\begin{equation}\label{meander}
 \tau_j\text{ \em is a null-homotopic loop based at a vertex.}
\end{equation}
Notice that for only finitely many $j\in J$ can $\tau_j$ be based at the same vertex~$v$. Indeed, given $j\in J$, let $n$ be the unique integer such that $j\in J_n\sm J_{n-1}$. Then either $j=0$, or $\tau_j$ is a segment of $\tau_0$ followed or preceded by~$\pi_n$, and hence $\pi_n$ is a pass through an edge at~$v$. Since $\tau_0$ contains only finitely many such passes, this can happen for only finitely many~$n$, and indices $j$ first appearing in $J_n$ for different $n$ are distinct.

Next, let us show the following:
\begin{txteq}\label{jumble}
The family of all simplices added or deleted in the construction, that is, of all~$\sigma_i$ ($i\in I$), all~$\tau_j$ ($j\in J$) and all~$\tau_{j(n)}$, is locally finite and hence good.
\end{txteq}
To prove~\eqref{jumble}, let $x$ be any point in~$G$. If $x$ is a vertex, let $E_x$ be the set of edges at~$x$; if $x\in\interior e$ for an edge~$e$, let $E_x := \{e\}$. Choose an open neighbourhood $U$ of $x$ contained in~$\bigcup E_x$. Since $\tau_0$ traverses each edge in $E_x$ only finitely often, only finitely many of the paths $\sigma_i$ ($i\in I$) meet~$U$. Similarly, any path $\tau_j$ with $j\in J$ that meets $U$ must be based at a vertex incident with an edge in~$E_x$. Since there are only finitely many such vertices, and at each only finitely many $\tau_j$ are based, only finitely many $\tau_j$ with $j\in J$ meet~$U$. Finally, consider a path~$\tau_{j(n)}$. This path traverses an edge (in~$\pi_n$), so if it meets $U$ it must also traverse an edge in~$E_x$ or adjacent to an edge in~$E_x$. Only finitely many of the passes $\pi_k$ traverse such edges. By~\eqref{passes}, any $\tau_j$ containing $\pi_k$ satisfies $j\in J_1\cup \ldots\cup J_{k-1}$, so $j(n)\in J_1\cup\ldots\cup  J_{k-1}$ for the largest such~$k$. Since this is a finite set and the map $n\mapsto j(n)$ is injective, only finitely many $n$ are such that $\tau_{j(n)}$ meets~$U$. This completes the proof of~\eqref{jumble}.

To complete the proof, we show that $z+b=\phi$ for $b:= \sum_{i\in I} b_i - \sum_{j\in J} \tau_j$, and in particular that $b\in B_1(G)$. By~\eqref{jumble}, the family of all simplices in~$b$ is good, so $b\in B_1(G)$ by \eqref{j} and~\eqref{meander}. Likewise, the family of all~$\sigma_i$ is good. Since
$$z + \sum_{i\in I} b_i =
     \sum_{i\in I} \sigma_i + \sum_{j\in J} \tau_j$$
by construction, we deduce that $z + b = \sum_{i\in I}\sigma_i = \phi$ as desired.
\end{proof}

We can now easily complete the proof that our function $f\colon H_1(G)\to \vCC(G)$ is injective. Consider any $[z]\in\ke f$. As $z\in Z_1(G)$, it has a standard representation as $z = \sum_{j\in J} z_j$ with all $z_j\in Z'_1(G)$. By Lemma~\ref{InfSubdiv}, there are $b_j\in B_1(G)$ ($j\in J$) such that $z_j + b_j = \phi_j$, where $\phi_j = \sum_{i\in I_j}\sigma_i$ is a chain of simplices each traversing exactly one edge, and these edges as well as the images of the simplices in $b_j$ lie in the image of~$z_j$. The fact that $z$ is a locally finite chain therefore implies that so are
 $$b:=\sum_{j\in J} b_j \emtext{~~and~~} \phi := \sum_{j\in J} \phi_j\,.$$
 Indeed, every $x\in G$ has an open neighbourhood $U$ that meets the images of simplices in $z_j$ for only finitely many~$j$; let $J_x$ be the set of those~$j$. Hence $U$~does not meet the images of any simplices in $b_j$ or $\phi_j$ for $j\notin J_x$. For each $j\in J_x$, we can find an open neighbourhood $U_j\sub U$ of $x$ that meets only finitely many simplices in $b_j$ or~$\phi_j$, because $b_j$ and $\phi_j$ are well-defined chains. The intersection of these finitely many $U_j$ thus is an open neighbourhood of $x$ that meets only finitely many simplices in $b$ or in~$\phi$, showing that $b$ and $\phi$ are well-defined chains.

For $I:=\bigcup_{j\in J} I_j$, we  thus have $Z_1\owns z + b = \phi = \sum_{i\in I}\sigma_i$, with $b\in B_1(G)$. Since $[z]\in\ke f$, we thus have $[\phi]\in\ke f$. Therefore the loops formed by the elementary cycles in $\phi = \sum_{i\in I}\sigma_i$ traverse, in total, each edge of $G$ equally often in both directions (Lemma~\ref{traverse}). Since each of the $\sigma_i$ traverses precisely one edge, we can thus pair them up into cancelling pairs $\sigma_i + \sigma_{i'}\in B_1(G)$, where $\sigma_i$ and $\sigma_{i'}$ traverse the same edge but in opposite directions. Hence $\phi = \sum_{i\in I}\sigma_i \in B_1(G)$, giving $z = \phi - b \in B_1(G)$ as desired.

\medbreak

We have thus shown that $f$ is an injective group homomorphism from $H_1(G)$ to $\vCC(G)$ whose restriction to $H'_1(G)$ still maps onto~$\vCC(G)$. Hence all these groups coincide, which is our second main result:

\begin{theorem}
The function $f$ is a group isomorphism between $H_1(G)$ and~$\vCC(G)$, as well as between $H'_1(G)$ and~$\vCC(G)$. In particular, $H'_1(G)=H_1(G)$.
\end{theorem}

\section{Acknowledgement}

We thank Laurent Bartholdi for pointing out to us the connection between $\C(G)$ and the Cech homology of~$|G|$, as described in Section~\ref{Cech}.

\bibliographystyle{amsplain}
\bibliography{collective}

\small
\parindent=0pt
\vskip2mm plus 1fill

\begin{tabular}{cc}
\begin{minipage}[t]{0.5\linewidth}
Reinhard Diestel\\
Mathematisches Seminar\\
Universit\"at Hamburg\\
Bundesstra\ss e 55\\
20146 Hamburg\\
Germany\\
\end{minipage} 
&
\begin{minipage}[t]{0.5\linewidth}
Philipp Spr\"ussel\\
Mathematisches Seminar\\
Universit\"at Hamburg\\
Bundesstra\ss e 55\\
20146 Hamburg\\
Germany\\
\end{minipage}
\end{tabular} 

\end{document}